\pdfoutput=1
\documentclass[12pt, a4paper, twoside,leqno]{amsart}
\usepackage[centering, totalwidth = 400pt, totalheight = 600pt]{geometry}
\usepackage{amssymb, amsmath, amsthm, graphicx}
\usepackage{enumerate, microtype, stmaryrd, url} 
\usepackage[latin1]{inputenc}
\usepackage{color}
\definecolor{darkgreen}{rgb}{0,0.45,0}
\usepackage[colorlinks,citecolor=darkgreen,linkcolor=darkgreen]{hyperref}
\usepackage[british]{babel}

\DeclareMathAlphabet{\mathbf}{OT1}{cmr}{b}{n}

\mathchardef\mhyphen="2D

\usepackage[arrow, matrix, tips, curve, graph, rotate]{xy}
\SelectTips{cm}{10}

\makeatletter
\def\matrixobject@{%
 \edef \next@{={\DirectionfromtheDirection@ }}%
 \expandafter \toks@ \next@ \plainxy@
 \let\xy@@ix@=\xyq@@toksix@
 \xyFN@ \OBJECT@}
\let\xy@entry@@norm=\entry@@norm
\def\entry@@norm@patched{%
 \let\object@=\matrixobject@
 \xy@entry@@norm }
\AtBeginDocument{\let\entry@@norm\entry@@norm@patched}
\makeatother

\newcommand{\twocong}[2][0.5]{\ar@{}[#2] \save ?(#1)*{\cong}\restore}
\newcommand{\twoeq}[2][0.5]{\ar@{}[#2] \save ?(#1)*{=}\restore}
\newcommand{\rtwocell}[3][0.5]{\ar@{}[#2] \ar@{=>}?(#1)+/l 0.2cm/;?(#1)+/r 0.2cm/^{#3}}
\newcommand{\urtwocell}[3][0.5]{\ar@{}[#2] \ar@{=>}?(#1)+/dl 0.2cm/;?(#1)+/ur 0.2cm/^{#3}}
\newcommand{\rtwocello}[3][0.5]{\ar@{}[#2] \ar@{=>}?(#1)+/l 0.2cm/;?(#1)+/r 0.2cm/_{#3}}
\newcommand{\ltwocell}[3][0.5]{\ar@{}[#2] \ar@{=>}?(#1)+/r 0.2cm/;?(#1)+/l 0.2cm/^{#3}}
\newcommand{\ltwocello}[3][0.5]{\ar@{}[#2] \ar@{=>}?(#1)+/r 0.2cm/;?(#1)+/l 0.2cm/_{#3}}
\newcommand{\dtwocell}[3][0.5]{\ar@{}[#2] \ar@{=>}?(#1)+/u 0.2cm/;?(#1)+/d 0.2cm/^{#3}}
\newcommand{\dltwocell}[3][0.5]{\ar@{}[#2] \ar@{=>}?(#1)+/ur 0.2cm/;?(#1)+/dl 0.2cm/^{#3}}
\newcommand{\drtwocell}[3][0.5]{\ar@{}[#2] \ar@{=>}?(#1)+/ul 0.2cm/;?(#1)+/dr 0.2cm/^{#3}}
\newcommand{\dthreecell}[3][0.5]{\ar@{}[#2] \ar@3{->}?(#1)+/u 0.2cm/;?(#1)+/d 0.2cm/^{#3}}
\newcommand{\utwocell}[3][0.5]{\ar@{}[#2] \ar@{=>}?(#1)+/d 0.2cm/;?(#1)+/u 0.2cm/_{#3}}
\newcommand{\dtwocelltarg}[3][0.5]{\ar@{}#2 \ar@{=>}?(#1)+/u 0.2cm/;?(#1)+/d 0.2cm/^{#3}}
\newcommand{\utwocelltarg}[3][0.5]{\ar@{}#2 \ar@{=>}?(#1)+/d 0.2cm/;?(#1)+/u 0.2cm/_{#3}}

\newcommand{\MM}{\raisebox{0.04em}{\scalebox{0.76}{$\ominus$}}}
\newcommand{\PP}{\raisebox{0.04em}{\scalebox{0.76}{$\oplus$}}}
\newcommand{\DD}{\raisebox{0.04em}{\scalebox{0.76}{$\odot$}}}

\newdir{(}{{}*!<0em,-.14em>-\cir<.14em>{l^r}}
\newdir{ (}{{}*!/-5pt/\dir{(}}
\newdir{ >}{{}*!/-5pt/\dir{>}}

\DeclareMathOperator{\colim}{colim}

\DeclareMathOperator{\im}{im}

\newcommand{\cat}[1]{\mathbf{#1}}
\newcommand{\op}{\mathrm{op}}

\newcommand{\thg}{{\mathord{\text{--}}}}

\newcommand{\dbr}[1]{\left\llbracket{#1}\right\rrbracket}

\newcommand{\defeq}{\mathrel{\mathop:}=}
\newcommand{\cd}[2][]{\vcenter{\hbox{\xymatrix#1{#2}}}}

\renewcommand{\phi}{\varphi}

\renewcommand{\O}{{\mathcal O}}
\renewcommand{\P}{{\mathcal P}}
\newcommand{\Q}{{\mathcal Q}}

\newcommand{\xtor}[1]{\cdl[@1]{{} \ar[r]|-{\object@{|}}^{#1} & {}}}

\makeatletter

\def\hookleftarrowfill@{\arrowfill@\leftarrow\relbar{\relbar\joinrel\rhook}}
\def\twoheadleftarrowfill@{\arrowfill@\twoheadleftarrow\relbar\relbar}
\def\leftbararrowfill@{\arrowdoublefill@{\leftarrow\mkern-5mu}\relbar\mapstochar\relbar\relbar}
\def\Leftbararrowfill@{\arrowdoublefill@{\Leftarrow\mkern-2mu}\Relbar\Mapstochar\Relbar\Relbar}
\def\leftringarrowfill@{\arrowdoublefill@{\leftarrow\mkern-3mu}\relbar{\mkern-3mu\circ\mkern-2mu}\relbar\relbar}
\def\lefttriarrowfill@{\arrowfill@{\mathrel\triangleleft\mkern0.5mu\joinrel\relbar}\relbar\relbar}
\def\Lefttriarrowfill@{\arrowfill@{\mathrel\triangleleft\mkern1mu\joinrel\Relbar}\Relbar\Relbar}

\def\hookrightarrowfill@{\arrowfill@{\lhook\joinrel\relbar}\relbar\rightarrow}
\def\twoheadrightarrowfill@{\arrowfill@\relbar\relbar\twoheadrightarrow}
\def\rightbararrowfill@{\arrowdoublefill@{\relbar\mkern-0.5mu}\relbar\mapstochar\relbar\rightarrow}
\def\Rightbararrowfill@{\arrowdoublefill@{\Relbar\mkern-2mu}\Relbar\Mapstochar\Relbar\Rightarrow}
\def\rightringarrowfill@{\arrowdoublefill@\relbar\relbar{\mkern-2mu\circ\mkern-3mu}\relbar{\mkern-3mu\rightarrow}}
\def\righttriarrowfill@{\arrowfill@\relbar\relbar{\relbar\joinrel\mkern0.5mu\mathrel\triangleright}}
\def\Righttriarrowfill@{\arrowfill@\Relbar\Relbar{\Relbar\joinrel\mkern1mu\mathrel\triangleright}}

\def\leftrightarrowfill@{\arrowfill@\leftarrow\relbar\rightarrow}
\def\mapstofill@{\arrowfill@{\mapstochar\relbar}\relbar\rightarrow}

\renewcommand*\xleftarrow[2][]{\ext@arrow 20{20}0\leftarrowfill@{#1}{#2}}
\providecommand*\xLeftarrow[2][]{\ext@arrow 60{22}0{\Leftarrowfill@}{#1}{#2}}
\providecommand*\xhookleftarrow[2][]{\ext@arrow 10{20}0\hookleftarrowfill@{#1}{#2}}
\providecommand*\xtwoheadleftarrow[2][]{\ext@arrow 60{20}0\twoheadleftarrowfill@{#1}{#2}}
\providecommand*\xleftbararrow[2][]{\ext@arrow 10{22}0\leftbararrowfill@{#1}{#2}}
\providecommand*\xLeftbararrow[2][]{\ext@arrow 50{24}0\Leftbararrowfill@{#1}{#2}}
\providecommand*\xleftringarrow[2][]{\ext@arrow 10{26}0\leftringarrowfill@{#1}{#2}}
\providecommand*\xlefttriarrow[2][]{\ext@arrow 80{24}0\lefttriarrowfill@{#1}{#2}}
\providecommand*\xLefttriarrow[2][]{\ext@arrow 80{24}0\Lefttriarrowfill@{#1}{#2}}

\renewcommand*\xrightarrow[2][]{\ext@arrow 01{20}0\rightarrowfill@{#1}{#2}}
\providecommand*\xRightarrow[2][]{\ext@arrow 04{22}0{\Rightarrowfill@}{#1}{#2}}
\providecommand*\xhookrightarrow[2][]{\ext@arrow 00{20}0\hookrightarrowfill@{#1}{#2}}
\providecommand*\xtwoheadrightarrow[2][]{\ext@arrow 03{20}0\twoheadrightarrowfill@{#1}{#2}}
\providecommand*\xrightbararrow[2][]{\ext@arrow 01{22}0\rightbararrowfill@{#1}{#2}}
\providecommand*\xRightbararrow[2][]{\ext@arrow 04{24}0\Rightbararrowfill@{#1}{#2}}
\providecommand*\xrightringarrow[2][]{\ext@arrow 01{26}0\rightringarrowfill@{#1}{#2}}
\providecommand*\xrighttriarrow[2][]{\ext@arrow 07{24}0\righttriarrowfill@{#1}{#2}}
\providecommand*\xRighttriarrow[2][]{\ext@arrow 07{24}0\Righttriarrowfill@{#1}{#2}}

\providecommand*\xmapsto[2][]{\ext@arrow 01{20}0\mapstofill@{#1}{#2}}
\providecommand*\xleftrightarrow[2][]{\ext@arrow 10{22}0\leftrightarrowfill@{#1}{#2}}
\providecommand*\xLeftrightarrow[2][]{\ext@arrow 10{27}0{\Leftrightarrowfill@}{#1}{#2}}

\makeatother

\newcommand*{\Cdot}[1][0.7]{%
  \mathop{\mathpalette{\CdotAux{#1}}\bullet}%
}
\newdimen\CdotAxis
\newcommand*{\CdotAux}[3]{%
  {%
    \settoheight\CdotAxis{$#2\vcenter{}$}%
    \sbox0{%
      \raisebox\CdotAxis{%
        \scalebox{#1}{%
          \raisebox{-\CdotAxis}{%
            $\mathsurround=0pt #2#3$%
          }%
        }%
      }%
    }%
    \dp0=0pt %
    \sbox2{$#2\bullet$}%
    \ifdim\ht2<\ht0 %
      \ht0=\ht2 %
    \fi
    \sbox2{$\mathsurround=0pt #2#3$}%
    \hbox to \wd2{\hss\usebox{0}\hss}%
  }%
}

\numberwithin{equation}{section}

\theoremstyle{plain}
\newtheorem{Thm}{Theorem}
\newtheorem{Prop}[Thm]{Proposition}

\newtheorem{Lemma}[Thm]{Lemma}

\theoremstyle{definition}
\newtheorem{Defn}[Thm]{Definition}

\newtheorem{Rk}[Thm]{Remark}

\newcommand{\Cat}{\cat{Cat}}

\def\wCat{\omega\mhyphen\Cat}

\newcommand{\atwo}{{\mathbf 2}}

\newcommand{\mm}[1][n]{\mathbin{\scriptstyle\circ_{#1}}}
\newcommand{\cell}[1]{{\boldsymbol{#1}}}

\begin{document}
\leftmargini=2em 
\title{Orientals and cubes, inductively}
\author{Mitchell Buckley}
\address{Department of Computing, 
  Macquarie University, NSW 2109, Australia} 
\email{mitchell.buckley@mq.edu.au}
\author{Richard Garner} 
\address{Department of Mathematics, 
  Macquarie University, NSW 2109, Australia}
\email{richard.garner@mq.edu.au} 

\thanks{The first author gratefully acknowledges the support of
  Macquarie University Research Centre funding; the second author
  acknowledges, with equal gratitude, the support of Australian Research
  Council Discovery Projects DP110102360 and DP130101969.}

\date{\today}
\begin{abstract}
  We provide direct inductive constructions of the orientals and the
  cubes, exhibiting them as the iterated cones, respectively, the
  iterated cylinders, of the terminal strict globular
  $\omega$-category.
\end{abstract}
\maketitle

\section{Introduction}
\label{sec:introduction}

A notorious aspect of the theory of weak higher dimensional categories
is the proliferation of models that have been proposed for the
notion~\cite{Leinster2002A-survey}; a major outstanding problem is
showing these different models to be suitably equivalent. Among the
technical challenges facing anyone looking to do so is one of
\emph{geometry}, since in every kind of model, one has a notion of
``$n$-cell'', but between models the shapes of these $n$-cells may
differ. There is a general agreement that ``$0$-cell'' and
``$1$-cell'' should mean ``point'' and ``arrow''; but beyond this, the
$n$-cells could be, among other things, simplicial, cubical or
globular in shape. In dimension two, for example, this means that
cells could take any of the following forms:
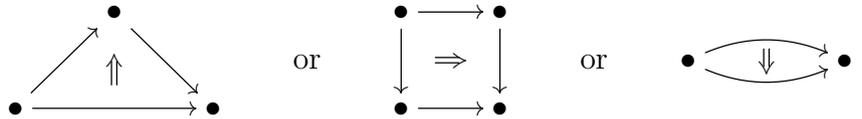
\begin{equation*}
  \cd{
    & {\bullet} \utwocell[0.6]{d}{} \ar[dr]^-{} \\ 
    {\bullet} \ar[ur]_-{} \ar[rr]_-{} & &
    {\bullet}\\
  } \qquad \text{or} \qquad 
  \cd{
    {\bullet} \ar[d]^-{} \ar[r]^-{} \rtwocell{dr}{} &
    {\bullet} \ar[d]_-{} \\
    {\bullet} \ar[r]^-{} &
    {\bullet}
  } \qquad \text{or} \qquad 
  \cd[@C-0.4em]{
    {\bullet} \ar@/^8pt/[rr]^-{} \ar@/_8pt/[rr]^-{} \ar@{}[rr] \ar@{=>}?+/u 0.16cm/;?+/d 0.16cm/ & &
    {\bullet}\rlap{ .}
  } 
\end{equation*}
In comparing two notions of model, then, a first step must always be
to describe a construction by which the basic cell-shapes of the one
kind of model may be built out of the cell-shapes of the other.

In the literature there are certain equivalences of models which have
been fully realised; one is the equivalence of strict globular
$\omega$-categories and strict cubical $\omega$-categories with
connections~\cite{Al-Agl2002Multiple}; another is the equivalence of
strict globular $\omega$-categories with \emph{complicial
  sets}~\cite{Verity2008Complicial}, whose geometry is simplicial in
nature. In particular, this means that the basic $n$-cell shapes of
these cubical and simplicial models can be realised as strict globular
$\omega$-categories, known respectively as the \emph{cubes} and the
\emph{orientals}. The orientals were constructed by Street
in~\cite{Street1987The-algebra}; his later \emph{parity
  complexes}~\cite{Street1991Parity} generalised the construction to
permit the realisation by strict globular $\omega$-categories of a
wide range of oriented polyhedra, including not only the orientals but
also the cubes.

Now, in undertaking the as-yet-unrealised task of relating simplicial,
cubical and globular models of \emph{weak} $\omega$-categories, it is
clear from the discussion above that a reasonable first step would be
the construction of suitably weakened analogues of the orientals or
cubes---that is, realisations of each $n$-simplex or $n$-cube as a
weak globular $\omega$-category. In this context, the theory of parity
complexes is of no use, since it makes free and implicit use of the
\emph{middle-four interchange} axioms present in a strict higher
category, but absent from a truly weak model; and so it is of interest
to find alternate constructions of the (strict) orientals and cubes
that may be more liable to adapt to the weak context.

In this paper, we describe one such alternate construction, which
builds the orientals and cubes inductively: the $(n+1)$st
oriental will be obtained as the \emph{cone} of the $n$th oriental,
and the $(n+1)$st cube as the \emph{cylinder} of the $n$th one. Here,
``cone'' and ``cylinder'' are certain operations on
$\omega$-categories to be introduced below; the nomenclature comes,
of course, from topology, where the \emph{cylinder} of a topological
space is its product with the interval, and the \emph{cone} the result
of collapsing one end of the cylinder to a point.

We have not yet attempted to adapt our inductive constructions from
strict to weak $\omega$-categories, but even without having done so,
we may still justify the worth of our inductive constructions from
another perspective: simplicity. The theory of parity complexes is
challenging, and the proof that any parity complex can be realised by
a strict $\omega$-category is both substantial and combinatorially
intricate. Our construction, by contrast, is relatively elementary,
and the proof of the equivalence with the original approach is
straightforward.

We have obtained further results concerning the $\omega$-categorical
cone and cylinder constructions; for reasons of space, the details of
these results are reserved for a future paper, but let us at least
outline them here. The first makes precise the analogy between our
cones and cylinders and the topological ones, by exhibiting the
cylinder of a strict globular $\omega$-category $X$ as its \emph{lax
  Gray tensor product}~\cite{Crans1995Pasting} $X \otimes \cat{2}$
with the arrow category, and exhibiting the cone of $X$ as the pushout
of the codomain inclusion $X \rightarrow X \otimes \cat{2}$ along the
unique map $X \rightarrow 1$. The second additional result has to do
with the \emph{freeness} of the orientals and cubes.
In~\cite{Street1991Parity}, a strict globular $\omega$-category is
called \emph{free} (also \emph{cofibrant}~\cite{Metayer2008Cofibrant})
when it admits a presentation by iteratively adjoining new $n$-cells
into existing $n$-cell boundaries. An important result
of~\cite{Street1991Parity} (the ``excision of extremals'' algorithm)
shows that the strict globular $\omega$-category on any parity
complex---so in particular, any oriental or cube---is free. Our second
additional result allows us to recover the freeness of the orientals
and cubes inductively, by showing that that both cone and cylinder
\emph{preserve freeness of $\omega$-categories}.

Beyond this introduction, this paper comprises the following parts.
Section~\ref{sec:background} describes some necessary background on
$\omega$-categories; Section~\ref{sec:cosl-cones-orient} introduces
our cone and cylinder constructions; Section~\ref{sec:orientals}
proves that the iterated cones of the terminal $\omega$-category are
the orientals, while Section~\ref{sec:cubes} proves that the iterated
cylinders of the terminal $\omega$-category are the cubes.
Appendix~\ref{sec:proofs-well-defin} gives proofs of well-definedness
deferred from Section~\ref{sec:cosl-cones-orient}.

\section{Background}
\label{sec:background}
In the rest of the paper, \emph{$\omega$-category} will mean
\emph{strict globular $\omega$-category}; in this section, we recall
those aspects of their theory necessary for our development. A
\emph{globular set} $X$ is a diagram of sets
\begin{equation*}
  \cd{
    \dots \ar@<3pt>[r]^-{s} \ar@<-3pt>[r]_-{t} &
    X_{n+1} \ar@<3pt>[r]^-{s} \ar@<-3pt>[r]_-{t} &
    X_n \ar@<3pt>[r]^-{s} \ar@<-3pt>[r]_-{t} & 
    \dots \ar@<3pt>[r]^-{s} \ar@<-3pt>[r]_-{t} & 
    X_0
  }
\end{equation*}
satisfying the globularity equations $ss = st$ and $ts = tt$. If $X$
is a globular set, then its \emph{$n$-cells} are the elements of
$X_n$; a pair of $n$-cells $x,y$ are \emph{parallel} if $n=0$ or if
$n>0$ and $(sx,tx) = (sy,ty)$. Given $0 \leqslant n < k$, we write $s_n, t_n$
for the maps $s^{k-n}, t^{k-n} \colon X_k \rightarrow X_n$, and for
$x \in X_k$ we call the parallel pair $(s_nx, t_nx)$ the
\emph{$n$-boundary} of $x$. We write $x \colon y \rightsquigarrow z$
to indicate that $(y,z)$ is the $n$-boundary of $x$; when $k = n+1$, we
may write $x \colon y \rightarrow z$ instead.

A \emph{small $\omega$-category} is a globular set $X$ with
\emph{identity} and \emph{composition} functions
\begin{equation*}
  i \colon X_n \rightarrow X_{n+1} 
  \qquad \text{and} \qquad
  \mm \colon X_k \,\mathop{{}_{t_n}\! \mathord \times_{s_n}} X_k \rightarrow X_k
\end{equation*}
for all $0 \leqslant n < k$, satisfying the following
three kinds of axioms. First, the \emph{source--target} axioms
that $s(ix) = t(ix) = x$ for all cells $x$ and that:
\begin{equation*} 
  s_k(x \mm y) =
  \begin{cases} s_k(x) \mm s_k(y) & \text{if $k > n$;}\\ s_k(y)
    & \text{if $k \leqslant n$,}
  \end{cases} \quad \ t_k(x \mm y) =
  \begin{cases} t_k(x) \mm t_k(y) & \text{if $k > n$;}\\ t_k(x)
    & \text{if $k \leqslant n$,}
  \end{cases}
\end{equation*} 
for all suitable cells $x$ and $y$. Second, the \emph{category} axioms
that $x \mm i(sx) = x = i(tx) \mm x$ and
$x \mm (y \mm z) = (x \mm y) \mm z$ for all suitable cells $x, y, z$.
Finally, the \emph{interchange} axiom that
$(x \mm y) \mm[k] (z \mm w) = (x \mm[k] z) \mm (y \mm[k] w)$ for all
$n < k$ and suitable cells $x,y,z,w$.

The \emph{dual} $X^\op$ of a globular set $X$ is the globular set
obtained by interchanging $s$ and $t$ at each stage; the \emph{dual}
$X^\op$ of a small $\omega$-category is given by the dual of the
underlying globular set of $X$ equipped with the same identities and
the reversed compositions at each dimension.

A map $f \colon X \rightarrow Y$ between globular sets comprises
functions $f_n \colon X_n \rightarrow Y_n$ satisfying
$sf_{n+1} = f_ns$ and $tf_{n+1} = f_nt$. An \emph{$\omega$-functor}
$f \colon X \rightarrow Y$ between $\omega$-categories is a map of
underlying globular sets which preserve composition and identities, in
the sense that $f(ix) = i(fx)$ and $f(x \mm y) = fx \mm fy$ for all suitable
cells $x$ and $y$. Of course, $\omega$-functors compose, and so we
have the category $\omega\text-\cat{Cat}$ of small $\omega$-categories
and $\omega$-functors. 

The category $\omega\text-\cat{Cat}$ has finite products, computed at
the level of underlying globular sets, and so we can consider the
category $(\omega\text-\cat{Cat})\text-\cat{Cat}$ of small
$\omega\text-\cat{Cat}$-\emph{enriched}~\cite{Kelly1982Basic}
categories; this is in fact equivalent to $\omega\text-\cat{Cat}$.
Indeed, given an $\omega$-category $X$, we obtain an
$\omega\text-\cat{Cat}$-category with object set $X_0$, with hom
$X(x,y)$ the $\omega$-category whose $n$-cells are the $(n+1)$-cells
$x \rightsquigarrow y$ in $X$, and with composition $\omega$-functors
$X(y,z) \times X(x,y) \rightarrow X(x,z)$ given by $\mm[0]$ in $X$.
Conversely, if $X$ is an $\omega\text-\cat{Cat}$-category, then there
is an $\omega$-category whose $0$-cells are the objects of $X$ and
whose $(n+1)$-cells are the disjoint union of the $n$-cells of each
$X(x,y)$, with composition $\mm[0]$ given by the composition maps of
$X$, and composition $\mm[n+1]$ given by $\mm$ in the appropriate
hom-$\omega$-category.

Using this identification, we obtain the standard enriched-categorical
notion of module (= profunctor) for $\omega$-categories. A \emph{right
  module} over an $\omega$-category $X$ comprises $\omega$-categories
$M(x)$ for each $x \in X_0$ together with $\omega$-functors
$m \colon M(y) \times X(x,y) \rightarrow M(x)$ for each $x,y \in X_0$
making each diagram:
\begin{equation*}
  \cd[@C-0.6em]{
    {M(z) \times X(y,z) \times X(x,y)} \ar[r]^-{m \times 1} \ar[d]_{1 \times m} &
    {M(y) \times X(x,y)} \ar[d]^{m} \\
    {M(z) \times X(x,z)} \ar[r]_-{m} &
    {M(x)}
  } \!\!\! 
  \cd[@!C@C-5em]{
    & {M(x) \times 1} \ar[dl]_-{1 \times i} \ar[dr]^-{\cong} \\
    {M(x) \times X(x,x)} \ar[rr]_-{m} & &
    {M(x)}
  }
\end{equation*}
commute in $\omega\text-\cat{Cat}$. A \emph{left module} over $X$ is
defined dually, while if $X$ and $Y$ are $\omega$-categories, then a
\emph{$Y$-$X$-bimodule} comprises $\omega$-categories $M(x,y)$ for
$x,y \in X_0 \times Y_0$ such that each $M(x,\thg)$ is a left
$Y$-module, each $M(\thg,y)$ is a right $X$-module, and each diagram
of the following form commutes:
\begin{equation*}
    \cd[@C-0.6em]{
    {Y(y,y') \times M(x,y) \times X(x',x)} \ar[r]^-{m \times 1} \ar[d]_{1 \times m} &
    {M(x,y') \times X(x',x)} \ar[d]^{m} \\
    {Y(y,y') \times M(x',y)} \ar[r]_-{m} &
    {M(x',y')}\rlap{ .}
  }
\end{equation*}

General enriched-categorical principles allow us to assign to any right
$X$-module $M$ a new $\omega$-category $\mathrm{coll}(M)$, called the
\emph{collage}~\cite{Street2004Cauchy} of $M$. As an
$\omega\text-\cat{Cat}$-category, $\mathrm{coll}(M)$ has object-set
$X_0 + \{\star\}$ and hom-$\omega$-categories:
\begin{equation*}
 \mathrm{coll}(M)(x,y) = \begin{cases} 
 X(x,y) & \text{if $x,y \in X_0$;} \\
 M(x) &\text{if $x \in X_0$ and $y = \star$;} \\
 \emptyset & \text{if $y \in X_0$ and $x = \star$;} \\
 1 &\text{if $x=y=\star$.}
\end{cases}
\end{equation*}
The non-trivial compositions in $\mathrm{coll}(M)$ are obtained from
composition in $X$ augmented by the action morphisms
$M(y) \times X(x,y) \rightarrow M(x)$. Dually, each right $X$-module
also has a collage, while if $M$ is a $Y$-$X$-bimodule, then its
collage $\mathrm{coll}(M)$ has object set $X_0 + Y_0$, hom-categories
\begin{equation*}
\mathrm{coll}(M)(u,v) = \begin{cases} 
 X(u,v) &\text{if $x,y \in X_0$;}\\
 \emptyset & \text{if $u \in Y_0$ and $v \in X_0$;}\\
 M(u,v) &\text{if $u \in X_0$ and $v \in Y_0$;}\\
 Y(u,v) & \text{if $x,y \in Y_0$,}
\end{cases}
\end{equation*}
and non-trivial compositions given by the composition morphisms 
of $X$ together with the left and right $M$-action morphisms
$M(x,y) \times X(x',x) \rightarrow M(x',y)$ and $Y(y,y') \times M(x,y)
\rightarrow M(x,y')$.

\section{Cones and cylinders}
\label{sec:cosl-cones-orient}
We now introduce the lax \emph{coslices} and \emph{slices} of an
$\omega$-category, and use them to define the basic \emph{cone} and
\emph{cylinder} constructions, whose iterated application will yield
the orientals and cubes. To simplify notation, it will be convenient
henceforth to adopt the following conventions. First, we assume that
$\mm$ binds more tightly than $\mm[k]$ whenever $n < k$. In other
words, we take it that:
\begin{equation*}
  x \mm y \mm[k] z \defeq (x \mm y) \mm[k] z \qquad \text{and} \qquad 
  x \mm[k] y \mm z \defeq x \mm[k] (y \mm z)\rlap{ ,}
\end{equation*}
and similarly for longer unbracketed composites. Second, we implicitly
identify any $k$-cell with the identity $(k+\ell)$-cell thereon where
necessary to make binary composition type-check. In other words, we
take it that
\begin{equation*}
  x \mm y \defeq x \mm i^\ell(y) \qquad \text{and} \qquad w \mm z \defeq i^\ell(w) \mm z
\end{equation*}
for all suitable $x \in X_{k+\ell}$ and $y \in X_k$ or
$w \in X_{k+\ell}$ and $z \in X_k$. We refer to the resultant
composite as the \emph{whiskering} of the $(k+\ell)$-cell by the
$k$-cell.

\begin{Defn}
  \label{def:2}
  If $X$ is an $\omega$-category and $a \in X_0$, then the \emph{lax
    coslice} $\omega$-category $a / X$ is defined as follows.
  \begin{itemize}
  \item $0$-cells $\cell x = (x, {\bar x})$ are pairs $x \in X_0$ and
    ${\bar x} \colon a \to x$.
  \vskip0.25\baselineskip
  \item $(n+1)$-cells $\cell x = (x, {\bar x})$ with $i$-boundary $(\cell
    m_i, \cell p_i)$ for $i \leqslant n$
    are given by pairs of the following form when $n$ is even:
  \begin{equation*}
    (x \colon m_n \to p_n, \ 
    {\bar x} \colon \bar p_{n-1}
    \mm[n-1] \cdots \bar p_3 \mm[3] \bar p_1 \mm[1] x \mm[0] \bar m_0 \mm[2] \bar m_2 \cdots \mm \bar m_n \to \bar p_n)\rlap{ ,}
  \end{equation*}
  and by pairs of the following form when $n$ is odd:
  \begin{equation*}
    (x \colon m_n \to p_n, \ 
    {\bar x} \colon \bar m_n \to \bar p_n \mm \cdots \bar p_3 \mm[3] \bar p_1 \mm[1] x \mm[0] \bar m_0 \mm[2] \bar m_2 \cdots \mm[n-1] \bar m_{n-1})\rlap{ .}
  \end{equation*}
\item If $\cell x$ and $\cell y$ satisfy $t_n(\cell x) = s_n(\cell y)$,
  with common $i$-boundary $(\cell m_i, \cell p_i)$ for each $i < n$,
  then $\cell y \mm \cell x$ is given by the following pair when $n$ is even:
  \begin{equation*}
    (y \mm x,\, \bar y \mm[n+1] \bar p_{n-1} \mm[n-1] \cdots \bar p_3 \mm[3] \bar p_1 \mm[1] s_{n+1}y \mm[0] \bar m_0 \mm[2] \bar m_2 \cdots \mm[n-2] \bar m_{n-2} \mm {\bar x})
  \end{equation*}
  and by the following pair when $n$ is odd:
  \begin{equation*}
    (y \mm x,\, \bar y \mm \bar p_{n-2} \mm[n-2] \cdots \bar p_3 \mm[3] \bar p_1 \mm[1] t_{n+1}x \mm[0] \bar m_0 \mm[2] \bar m_2 \cdots \mm[n-1] \bar m_{n-1} \mm[n+1] {\bar x})\rlap{ .}
  \end{equation*}
\item The identity $(n+1)$-cell on an $n$-cell $(x,{\bar x})$ is $(ix,
  i{\bar x})$.
\end{itemize}
We write $\pi \colon a/X \rightarrow X$ for the $\omega$-functor
defined by $\pi(x, {\bar x}) = x$. 

Dually, for any $b \in X_0$, we define the \emph{lax slice}
$\omega$-category $X/b$ to be $(b/X^\op)^\op$; explicitly, this means
that:
  \begin{itemize}
  \item $0$-cells $\cell x = (x, \hat x)$ are pairs $x \in X_0$ and
    $\hat x \colon x \to b$.
  \vskip0.25\baselineskip
  \item $(n+1)$-cells $\cell x = (x, \hat x)$ with $i$-boundary $(\cell
    m_i, \cell p_i)$ for $i \leqslant n$
    are given by pairs of the following form when $n$ is even:
  \begin{equation*}
    (x \colon m_n \to p_n, \ 
    \hat x \colon \hat m_n \rightarrow \hat p_n
    \mm[n] \cdots \hat p_2 \mm[2] \hat p_0 \mm[0] x \mm[1] \hat m_1 \mm[3] \hat m_3 \cdots \mm[n-1] \hat m_{n-1})\rlap{ ,}
  \end{equation*}
  and by pairs of the following form when $n$ is odd:
  \begin{equation*}
    (x \colon m_n \to p_n, \ 
    \hat x \colon \hat p_{n-1} \mm[n-1] \cdots \hat p_2 \mm[2] \hat p_0 \mm[0] x \mm[1] \hat m_1 \mm[3] \hat m_3 \cdots \mm \hat m_n \to \hat p_n)\rlap{ ,}
  \end{equation*}
\end{itemize}
with composition being given dually to that in $a / X$. Combining the
preceding two constructions, if $X$ is an $\omega$-category and
$a, b \in X_0$, then the \emph{lax bislice} $a/X/b$ is the pullback of
$\omega$-categories $a/X \times_X X/b$. Explicitly, this means that
objects in $a / X / b$ are triples $(x, \bar x, \hat x)$ where
$x \in X_0$ and $\bar x \colon a \rightarrow x$ and
$\hat x \colon x \rightarrow b$, and similarly for higher cells.
\end{Defn}
In pasting notation, a $1$-cell
$(f, \bar f) \colon (x, \bar x) \rightarrow (y, \bar y)$ in $a/X$ is
given by a lax-commutative triangle as on the left of the following
diagram; the entirety of this same diagram depicts a $2$-cell
$(\alpha, \bar \alpha) \colon (f, \bar f) \rightarrow (g, \bar g)$ in
$a/X$.
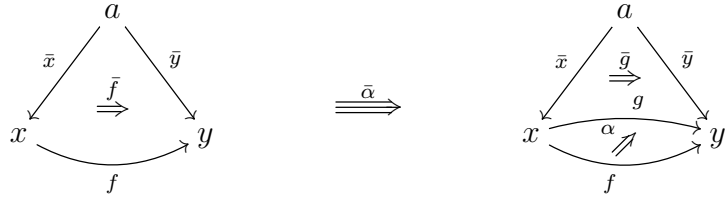
\begin{equation*}
  \cd[@-0.2em@R+1em]{
    & {a} \ar[dl]_-{\bar x} \ar[dr]^-{\bar y} \rtwocell[0.8]{d}{\bar f} \\
    {x} \ar@/_1em/[rr]_-{f} & &
    {y}
  }\qquad \quad \cd{{} \ar@3[r]^{\bar \alpha} & {}} \qquad \quad 
  \cd[@-0.2em@R+1em]{
    & {a} \ar[dl]_-{\bar x} \ar[dr]^-{\bar y} \rtwocell[0.55]{d}{\bar g} \urtwocell[1.075]{d}{\alpha} \\
    {x} \ar@/^0.5em/[rr]^(0.57){g} \ar@/_1em/[rr]_(0.43){f} & &
    {y}
  }
\end{equation*}
Correspondingly, a $1$-cell in the lax slice $X/b$ is as on the right
in:
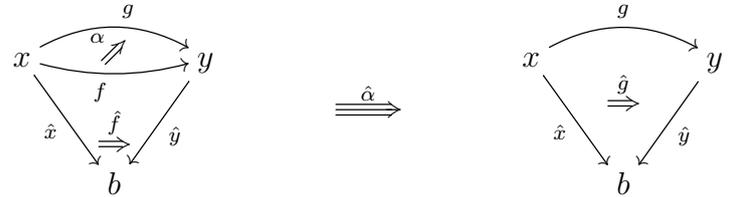
\begin{equation*}
  \cd[@-0.2em@R+1em]{
    {x} \ar[dr]_-{\hat x} \ar@/^1em/[rr]^(0.57){g} \ar@/_0.5em/[rr]_(0.43){f} & &
    {y} \ar[dl]^-{\hat y} \\ &
    {b} \rtwocell[0.33]{u}{\hat f} \urtwocell[1.05]{u}{\alpha}
  }
  \qquad \quad \cd{{} \ar@3[r]^{\hat \alpha} & {}} \qquad \quad 
  \cd[@-0.2em@R+1em]{
    {x} \ar@/^1em/[rr]^-{g} \ar[dr]_-{\hat x} & \rtwocell[0.35]{d}{\hat g} &
    {y\rlap{ ,}} \ar[dl]^-{\hat y} \\ &
    {b}
  }
\end{equation*}
while the whole diagram depicts a typical $2$-cell.

It is not immediate that the lax coslice (and hence slice)
$\omega$-categories are well-defined. One way to show this is to view
their construction as a particular case of the \emph{Grothendieck
  construction} for $\omega$-categories
of~\cite[Section~4.1]{Warren2008Homotopy}. This construction assigns
to each left module $M$ over an $\omega$-category $X$, an
$\omega$-functor (an ``opfibration'') $\int\! M \rightarrow X$;
applying it to the representable left module $X(a, \thg)$ yields
$\pi \colon a/X \rightarrow X$, so that the well-definedness of $a/X$
is a consequence of~\cite[Proposition~4.6]{Warren2008Homotopy}.
However, one may also prove well-definedness directly; so as to have a
self-contained presentation, we give this proof as
Proposition~\ref{prop:1} below.\looseness=-1

\begin{Defn}
  \label{def:5}
  Let $X$ be an $\omega$-category and $a,b \in X_0$. We define the
  $\omega$-functor $\Cdot \colon a/X \times X(b,a) \rightarrow b/X$ to
  have action on cells given inductively as follows:
\begin{itemize}
\item On $0$-cells, $\cell x \Cdot h = (x, {\bar x} \mm[0] h)$;
\item On $(n+1)$-cells, $\cell x \Cdot h = (x, {\bar x} \mm[0] h) \colon
  s\cell x \Cdot sh \rightarrow t\cell x \Cdot th$.
\end{itemize}
\end{Defn}
Once again, it is not immediate that this gives a well-defined
$\omega$-functor; we verify this in Proposition~\ref{prop:5} below.
Thereafter, it is immediate that these actions satisfy the necessary
associativity and unit axioms to give a right $X$-module $(\thg)/X$.
By duality, we obtain from the lax slice categories $X / b$ a left
module $X / (\thg)$, and from the lax bislice categories an
$X$-$X$-bimodule $(\thg) / X / (\thg)$ whose value at $a,b$ is $a/X/b$
with right $X$-action inherited from $a/X$ and left $X$-action
inherited from $X/b$. Using these modules, we may now give:

\begin{Defn}
  \label{def:4}
  If $X$ is an $\omega$-category, then $s(X)$, the \emph{cone under
    $X$} is the collage of the right $X$-module $(\thg)/X$. The
  \emph{cone over $X$}, $\bar s(X)$, is the collage of the left
  $X$-module $(\thg)/X$, while the \emph{cylinder on $X$}, $c(X)$, is
  the collage of the $X$-$X$-bimodule $(\thg) / X / (\thg)$.
\end{Defn}

\section{Orientals} \label{sec:orientals} In this section, we prove
our first main result, identifying the iterated cones of the terminal
$\omega$-category with the \emph{orientals}
of~\cite{Street1987The-algebra}. We begin by recalling the definition
of the orientals, following the presentation
of~\cite{Street1991Parity}. For any natural numbers $n$ and $j$, we
write $[n]$ for the set $\{0, \dots, n\}$, and write $[n]_{j}$ for
the set of order-preserving injections $[j] \rightarrowtail [n]$. If
$j > 0$ and $a \in [n]_{j}$, then we define the sets of \emph{even}
and \emph{odd} faces $a^+, a^- \subset [n]_{j-1}$ by
\[a^+ = \{a\delta_{2i} : 0 \leqslant 2i \leqslant j\}
\quad \text{and} \quad a^- = \{a\delta_{2i+1} : 0 \leqslant 2i+1
\leqslant j\}\rlap{ ,}
\]
where here $\delta_{k} \colon [j-1] \rightarrowtail [j]$ is the unique
order-preserving injection for which $k \notin \im \delta_k$. If
$\xi \subset [n]_{j}$, then we write
$\xi^- = \bigcup_{a \in \xi} a^-$ and
$\xi^+ = \bigcup_{a \in \xi} a^+$. We may write elements
$a \in [n]_{j}$ as increasing lists $(a_0 \cdots a_j)$ of elements
in $[n]$; with this notation, we have for example:
\begin{equation*}
  \{(135),(125)\}^{-} = \{(15)\} \quad \text{and} \quad 
  \{(135),(125)\}^{+} = \{(35),(13),(12),(25)\}\rlap{ .}
\end{equation*}

\begin{Defn} \label{def:orientals} The $n$th oriental $\O(n)$ is the
  strict $\omega$-category defined as follows.
  \begin{itemize}
  \item $0$-cells are natural numbers $i \in \{0, \dots, n\}$; we
    identify $i$ with the singleton subset $\{(i)\}$ of $[n]_0$.\vskip0.25\baselineskip
  \item $(j+1)$-cells $\xi$ with successive $k$-boundaries
    $(\mu_k, \pi_k)$ for each $k \leqslant j$ are finite subsets
    $\xi \subset [n]_{j+1}$ such that:\vskip0.25\baselineskip
 \begin{enumerate}[(i)]
 \item If $a \neq b \in \xi$ then
   $a^+ \cap b^+ = a^- \cap b^- = \emptyset$;
 \item $\pi_j = (\mu_j \cup {\xi}^+) \setminus {\xi}^-$ and
   $\mu_j = (\pi_j \cup {\xi}^-) \setminus {\xi}^+$ in $[n]_{j}$.
 \end{enumerate}
\end{itemize}
Given $(j+1)$-cells $\xi \colon x \to y$ and $\zeta \colon y \to z$,
their $\mm[j]$-composite is $\zeta \cup \xi \colon x \to z$; while if
the $(j+1)$-cells $\xi \colon x \to y$ and $\zeta \colon w \to z$
satisfy $t_i(\xi) = s_i(\zeta)$ for a fixed $i < j$ then their
$\mm[i]$-composite is
$\zeta \cup \xi \colon w \mm[i] x \to z \mm[i] y$. The identity
$(j+1)$-cell on the $j$-cell $x$ is given by
$\emptyset \colon x \rightarrow x$. (In particular, 
\emph{whiskering} a $k$-cell of $\O(n)$ by a lower-dimensional
cell does not change its $k$-dimensional part).
\end{Defn}

As with the coslices of Section~\ref{sec:cosl-cones-orient},
it is by no means immediate that the orientals are well-defined
$\omega$-categories; the problem is showing that the well-formedness
and movement conditions are stable under composition, and it is one of
the main theorems of~\cite{Street1987The-algebra} that this is so.
Following~\cite{Street1991Parity}, when a set $\xi \subset [n]_{j}$
satisfies the condition in (i) above, we say that $\xi$ is
\emph{well-formed}, and when it satisfies the two conditions in (ii),
we say that \emph{$\xi$ moves $\mu_j$ to $\pi_j$}, and write
$\xi \colon \mu_j \rightarrowtriangle \pi_j$; we refer to
the two conditions involved as the \emph{first} and \emph{second
  movement conditions}.

Before giving our first main result, we recall from~\cite[\S
2]{Street1987The-algebra} a useful characterisation of the $1$-cells
of the orientals:
\begin{Lemma}
  \label{lem:1}
  Any $1$-cell $\xi \colon i \rightarrow j$ in $\O(n)$ is either
  $\emptyset \colon i \rightarrow i$ or takes the form
  $\xi = \{(k_0\, k_1), (k_1\, k_2), \dots, (k_{r-1}\, k_r)\}$ for
  $i = k_0 < \dots < k_r = j$ in $[n]$.
\end{Lemma}
\begin{proof}
  If $i \neq j$, then by movement we have
  $i \in \xi^- \setminus \xi^+$ and $j \in \xi^+ \setminus \xi^-$ and
  for all $k \neq i, j$ that $k \in \xi^-$ iff $k \in \xi^+$.
  By well-formedness, it follows that the values $i$ and $j$ appear
  \emph{exactly} once in $\xi$, as odd and even faces respectively,
  and all other values appear exactly twice, as an even and odd face
  respectively. This gives the required form; a similar argument shows
  that, when $i = j$, the only possibility is $\xi = \emptyset$.
\end{proof}

\begin{Thm} \label{thm:orientals} The $n$th cone $s^n(1)$ under the
  terminal $\omega$-category is isomorphic to the $n$th oriental
  $\O(n)$.
\end{Thm}
\begin{proof}
  We prove this by induction on $n$, simultaneously with the result
  that:
  \begin{equation}
    \label{eq:13}
    \text{if $\xi \colon x \rightarrow y$ and $\zeta \colon y \rightarrow z$ are $(j+1)$-cells in $\O(n)$, then $\xi \cap \zeta = \emptyset$.}
  \end{equation}
  The case $n=0$ is clear. For the inductive step, we assume the
  result for $n$, and begin by showing $s(\O(n)) \cong \O(n+1)$.
  Removing $n+1$ from $\O(n+1)$ or $\star$ from $s(\O(n))$ yields in
  both cases $\O(n)$, and in both cases, the only maps from this
  removed object are identities. It thus suffices to find
  $\omega$-isomorphisms
  $\varphi_i \colon s(\O(n))(i, \star) = i / \O(n) \to \O(n+1)(i,
  n+1)$
  which are compatible with composition, in the sense that for each
  $(\cell x, h) \in j / \O(n) \times \O(n)(i, j)$, we have
  $\varphi_i(\cell x \Cdot h) = \varphi_j(\cell x) \mm[0] h$.

  First we introduce some notation. Given
  $a = (a_0 \cdots a_j) \in [n]_{j}$, we write $a^\vee$ for
  $(a_0 \cdots a_j\ n+1) \in [n+1]_{j+1}$, and given
  $\xi \subset [n]_{j}$, we write $\xi^\vee$ for
  $\{a^\vee : a \in \xi\}$. Note that for any $\xi \subset [n]_{j}$
  and $j > 0$, we have:
  \begin{equation}
    (\xi^{\vee})^+ =
    \begin{cases}
      (\xi^+)^{\vee} & \text{if $j$ even;}\\
      (\xi^+)^{\vee} \cup \xi & \text{if $j$ odd,}
    \end{cases} \quad (\xi^{\vee})^- =
    \begin{cases}
      (\xi^-)^{\vee} \cup \xi & \text{if $j$ even;}\\
      (\xi^-)^{\vee} & \text{if $j$ odd.}
    \end{cases}\label{eq:10}
  \end{equation}
  We now define $\varphi_i \colon i / \O(n) \to \O(n+1)(i, n+1)$ on
  cells of all dimension by
  \[
    \phi_i(\cell x) =
    \begin{cases}
      {\bar x} \cup x^\vee \colon i \to n+1 & \text{for $\cell x$ a $0$-cell;}\\
      {\bar x} \cup x^\vee \colon\phi_i(\cell m) \to \phi_i(\cell p) &
      \text{for $\cell x \colon \cell m \rightarrow \cell p$ a
        $(j+1)$-cell.}
    \end{cases}
  \]

We will show by induction on dimension that this assignation is
well-defined and bijective. For the base case, we use
Lemma~\ref{lem:1}. Any $0$-cell $(x, \bar x \colon i \rightarrow x)$
of $i / \O(n)$ has either $i = x$ and $\bar x = \emptyset$---in which
case $\varphi_i(\cell x) = \{(i\,\, n+1)\} \colon i \rightarrow n+1$
is well-defined---or has $i < x$ and
$\bar x = \{(i\, k_1), \dots, (k_{r-1}\, x)\}$---in which case
$\phi_i(\cell x) = \{(i\, k_1), \dots, (k_{r-1}\, x), (x\,\, n+1)\}
\colon i \rightarrow n+1$
is again well-defined. In fact, by Lemma~\ref{lem:1}, \emph{any}
$\xi \colon i \rightarrow n+1$ in $\O(n+1)$ is uniquely of one of the
two forms just listed, so that $\varphi_i$ is a bijection on
$0$-cells.

Suppose now that we have shown that $\varphi_i$ is well-defined and
bijective on all cells up to dimension $j$; then for any parallel pair
of $j$-cells $(\cell m_j, \cell p_j)$ in $i / \O(n)$ with successive
boundaries $(\cell m_k, \cell p_k)$ for $k<j$, we will show that
$\phi_i$ gives a well-defined bijection between cells
$\cell x \colon \cell m_j \rightarrow \cell p_j$ and ones
$\xi \colon \phi_i(\cell m_j) \rightarrow \phi_i(\cell p_j)$. We
consider only the case where $j$ is odd; the even case is identical in
form, and so omitted. Observe first that the operation on subsets
\begin{equation}
  \label{eq:11}
  \begin{aligned}
    \P [n]_{j+1} \times \P [n]_{j+2} & \rightarrow \P [n+1]_{j+2}  \\
    (x, \bar x) & \mapsto \bar x \cup x^\vee
  \end{aligned}  
\end{equation}
underlying $\phi_i$'s action on $(j+1)$-cells is bijective. A cell
$\cell x \colon \cell m_j \rightarrow \cell p_j$ is an element
$(x, \bar x)$ of the domain of~\eqref{eq:11} satisfying the three
conditions that:
\begin{equation*}
  \text{(i) $x$ and $\bar x$ are well-formed;} \qquad 
  \text{(ii)  $x \colon m_j\rightarrowtriangle p_j$;} \qquad 
  \text{(iii) $\bar x \colon \bar m_j \rightarrowtriangle \bar p_j \cup x$,}
\end{equation*}
while a cell
$\phi_i(\cell m_j) \rightarrow \phi_i(\cell p_j)$ is an element
$\bar x \cup x^\vee$ of the codomain such that:
\begin{equation*}
\text{(iv) $\bar x \cup x^\vee$ is well-formed;} \qquad 
\text{(v) $\bar x \cup x^\vee \colon \bar m_j \cup m_j^\vee \rightarrowtriangle \bar p_j
\cup p_j^\vee$.}
\end{equation*}
Thus to check well-definedness and bijectivity of $\varphi_i$ on
$(j+1)$-cells, it suffices to show that (i)--(iii) are equivalent to
(iv) \& (v). 

Now, if $\bar x \cup x^\vee$ is well-formed then clearly so is
$\bar x$, but in fact also $x$, as if $a \neq b \in x$ shared an even
or odd face, then so would $a^\vee \neq b^\vee$ in
$x^\vee \subset \xi$. Thus (iv) implies (i). Conversely, if $x$ and
$\bar x$ are well-formed, then both components of $\bar x \cup x^\vee$
are individually well-formed, while if $a \in {\bar x}$ and
$b^\vee \in x^\vee$, then clearly $a$ and $b^\vee$ share no even faces
since $n+1 \notin a$, and could only share an odd face if $b \in x$
were an odd face of $a \in {\bar x}$; and this is impossible if
${\bar x} \colon {\bar m_j} \rightarrowtriangle {\bar p_j} \cup x$ since then
$({\bar p_j} \cup x) \cap {\bar x}^- = \emptyset$. So (i) and (iii)
imply (iv).

Turning now to the movement conditions, we have 
$(\bar x \cup x^\vee)^+ = {\bar x}^+ \cup (x^{+})^\vee$ and
$(\bar x \cup x^\vee)^- = {\bar x}^- \cup x \cup (x^-)^\vee$ by~\eqref{eq:10}; so (v)
is equivalent to:
 \begin{align*}
 ({\bar p_j} \cup p_j^{\vee}) & = ({\bar m_j} \cup m_j^{\vee})
 \cup ({\bar x}^+ \cup (x^{+})^\vee) \setminus ({\bar x}^- \cup x \cup (x^{-})^\vee
 )  \\
({\bar m_j} \cup m_j^{\vee}) &= ({\bar p_j} \cup p_j^{\vee})
 \cup ({\bar x}^- \cup x \cup (x^{-})^\vee) \setminus ({\bar x}^+ \cup
 (x^{+})^\vee)\rlap{ ;}
 \end{align*}
 now as the terms which are under $(\thg)^\vee$ are disjoint from
 those which are not, and $(\thg)^\vee$ is a bijection, this is
 equivalent to the four conditions:
 \begin{equation}
 \begin{aligned}
   p_j & = m_j \cup x^+ \setminus x^- & \qquad {\bar p_j} &= {\bar m_j} \cup {\bar x}^+ \setminus ({\bar x}^- \cup x) \\
m_j & = p_j \cup x^- \setminus x^+ & \qquad \bar m_j & = \bar p_j \cup {\bar x}^- \cup x \setminus {\bar x}^+\rlap{ .}
 \end{aligned}\label{eq:12}
\end{equation}
The left two are precisely (ii), and the lower right is the second
movement condition for (iii). The upper right will imply the first
movement condition
$\bar p_j \cup x = \bar m_j \cup {\bar x}^+ \setminus \bar x^-$ for (iii)
so long as $\bar x^- \cap x = \emptyset$; but this is certainly the
case if $\bar x \cup x^\vee$ is well-formed, as if $a \in \bar x$ had
an odd face $b$ in $x$, then $a \neq b^\vee$ would share an odd face
in $\bar x \cup x^\vee$. So (iv) and (v) imply (ii) and (iii).
Finally, since (v) is equivalent to the conditions in~\eqref{eq:12},
it will follow from (ii) and (iii) so long as we know that $x \cap
\bar p_j = \emptyset$. Now, observe that $\bar x$ is a cell
\begin{equation*}
  \bar m_j \rightarrow \bar p_j \mm[j] \bar p_{j-2} \mm[j-2] \cdots \bar p_1 \mm[1] x \mm[0] \bar m_0  \cdots \mm[j-1] \bar m_{j-1}
\end{equation*}
so that in particular, the cells $\bar p_j$ and
$\bar p_{j-2} \mm[j-2] \cdots \bar p_1 \mm[1] x \mm[0] \bar m_0 \cdots
\mm[j-1] \bar m_{j-1}$
of $\O(n)$ are $j$-composable; applying the inductive instance
of~\eqref{eq:13} for $\O(n)$ we conclude that
$x \cap \bar p_j = \emptyset$ as required.

This show that each $\varphi_i$ is a bijective map on cells of all
dimension; it remains only to show $\omega$-functoriality and
compatibility with the actions by $\Cdot$ and $\mm[0]$. If $\cell x$
and $\cell y$ are $(j+1)$-cells of $i /\O(n)$ with
$t_\ell(\cell x) = s_\ell(\cell y)$ and with common boundary
$(\cell m_k, \cell p_k)$ for each $k < \ell$, then:
  \begin{align*}
  \phi_i(\cell y \mm[\ell] \cell x) 
  & = \phi_i(y \mm[\ell] x,\, \bar y \mm[\ell+1] \bar p_{\ell-1} \mm[\ell-1] \cdots \bar p_1 \mm[1] s_{\ell+1}y \mm[0] \bar m_0 \cdots \mm[\ell-2] \bar m_{\ell-2} \mm[\ell] {\bar x})\\
  & = \phi_i(y \cup x, \bar y \cup \bar x)
  = (\bar y \cup {\bar x}) \cup (y \cup x)^\vee \\ 
  & = (x^\vee \cup {\bar x}) \cup (y^\vee \cup {\bar y})  = \phi_i(\cell y) \mm \phi_i(\cell x)
 \end{align*}
 when $\ell$ is even, and correspondingly when $\ell$ is odd. As for
 identity morphisms, we have
 $\phi_i(i\cell x) = \phi_i(ix, i{\bar x}) = \emptyset \cup
 \emptyset^\vee = \emptyset = i\phi_i(\cell x)$;
 so $\phi_i$ is $\omega$-functorial as required. To show compatibility
 of the $\varphi_i$'s with composition, we argue similarly that
 $ \phi_i(\cell x \Cdot h) = \phi_i(x, {\bar x} \mm[0] h) = ({\bar x}
 \cup h) \cup x^\vee = ({\bar x} \cup x^\vee) \cup h = \phi_j(\cell x)
 \mm[0] h$.

 This proves that $s(\O(n)) \cong \O(n+1)$, and it remains only to
 derive~\eqref{eq:13} for $\O(n+1)$. The case of $0$-composable
 $1$-cells is easy from Lemma~\ref{lem:1}, while any pair of
 $(j+1)$-composable $(j+2)$-cells must live in some
 hom-$\omega$-category of $\O(n+1)$; the only new case to consider is
 that of $\O(n+1)(i, n+1) \cong i / \O(n)$. For this, let
 $\cell x \colon \cell a \rightarrow \cell b$ and
 $\cell y \colon \cell b \rightarrow \cell c$ be $(j+1)$-cells in
 $i / \O(n)$ with common boundary $(\cell m_k, \cell p_k)$ for all
 $k < j$. We have $x \colon a \rightarrow b$ and
 $y \colon b \rightarrow c$ in $\O(n)$, whence $x \cap y = \emptyset$
 by~\eqref{eq:13} for $\O(n)$; moreover, assuming $j$ is even, we have
 that $\bar y$ and
 $\bar p_{j-1} \mm[j-1] \cdots \bar p_1 \mm[1] y \mm[0] \bar m_0
 \cdots \mm[j-2] \bar m_{j-2} \mm[j] {\bar x}$
 are $(j+1)$-composable $(j+2)$-cells, whence
 $\bar x \cap \bar y = \emptyset$ again by~\eqref{eq:13} for $\O(n)$;
 a similar argument shows $\bar x \cap \bar y = \emptyset$ when $j$ is
 odd. We conclude that the composable pair
 $\phi_i(\cell x) = \bar x \cup x^\vee$ and
 $\phi_i(\cell y) = \bar y \cup y^\vee$ satisify
 $(\bar x \cup x^\vee) \cap (\bar y \cup y^\vee) = \emptyset$, as
 required.
\end{proof}

\begin{Rk}
  \label{rk:1}
  The condition~\eqref{eq:13} on composition of cells in orientals is
  proved by Street in~\cite[Theorem~3.12]{Street1987The-algebra}; by
  not simply quoting his result, we have avoided using any aspect of
  the theory of orientals beyond the basic definitions, and this
  allows our main theorem to provide an alternative and simpler proof
  that the orientals do indeed have a well-defined composition.
  Arguing inductively, once we know that $\O(n)$ is well-defined, then
  so too is $s(\O(n))$; now transporting across the isomorphism of
  globular sets $s(\O(n)) \cong \O(n+1)$ shows that composition in
  $\O(n+1)$ is also well-defined.
\end{Rk}

\section{Cubes}\label{sec:cubes}
We now turn to our second main result, which will identify the
iterated cylinders on the terminal $\omega$-category with the
\emph{cubes}. We begin by recalling their definition, following again
the presentation of~\cite{Street1991Parity}. Given natural numbers $n$
and $j$, we write $\dbr{n}$ for the set of strings of length $n$ in
the symbols $\MM$, $\DD$, and $\PP$, and write $\dbr{n}_j$ for the
subset of such strings in which the symbol $\DD$ appears exactly $j$
times. If $j > 0$ and $a \in \dbr{n}_j$, then we define the sets
$a^-, a^+ \subset \dbr{n}_{j-1}$ of \emph{odd} and \emph{even} faces by:
\begin{equation*}
  a^- = \{a\delta_i^- : 1 \leqslant i \leqslant j\}\qquad \text{and} \qquad 
  a^+ = \{a\delta_i^+ : 1 \leqslant i \leqslant j\}
\end{equation*}
where $a\delta^-_i$ denotes the string obtained from $a$ by replacing
the $i$th occurence of $\DD$ therein by either $\MM$ or $\PP$
according as $i$ is odd or even, and where $a\delta^+_i$ denotes
similarly the string obtained by replacing the $i$th $\DD$ by either
$\PP$ or $\MM$ according as $i$ is odd or even. Like before, for any
$\xi \subset \dbr{n}_j$ we define $\xi^- = \bigcup_{a \in \xi} a^-$
and $\xi^+ = \bigcup_{a \in \xi} a^+$; with this notation, we have,
for example, that:
\begin{equation*}
  \{\MM\DD\DD\}^+ = \{\MM\PP\DD, \MM\DD\MM\} \qquad \text{and} \qquad \{\MM\DD,\DD\MM\}^- = \{\MM\MM\}\rlap{ .}
\end{equation*}

\begin{Defn} \label{def:cubes}
The $n$th cube $\Q(n)$ is 
is the strict $\omega$-category defined as follows.
  \begin{itemize}
  \item $0$-cells are elements of $\dbr{n}_0$: strings of length
    $n$ of $\MM$'s and $\PP$'s. We identify each such string
    with the corresponding singleton subset of $\dbr{n}_0$.
\vskip0.25\baselineskip
  \item $(j+1)$-cells $\xi$ with successive $k$-boundaries
    $(\mu_k, \pi_k)$ for each $k \leqslant j$ are finite subsets
    $\xi \subset \dbr{n}_{j+1}$ such that:\vskip0.25\baselineskip
 \begin{enumerate}[(i)]
 \item If $a \neq b \in \xi$ then
   $a^+ \cap b^+ = a^- \cap b^- = \emptyset$;
 \item $\pi_j = (\mu_j \cup {\xi}^+) \setminus {\xi}^-$ and
   $\mu_j = (\pi_j \cup {\xi}^-) \setminus {\xi}^+$ in $\dbr{n}_j$.
 \end{enumerate}
\end{itemize}
Given $(j+1)$-cells $\xi \colon x \to y$ and $\zeta \colon y \to z$,
their $\mm[j]$-composite is $\zeta \cup \xi \colon x \to z$; while if
the $(j+1)$-cells $\xi \colon x \to y$ and $\zeta \colon w \to z$
satisfy $t_i(\xi) = s_i(\zeta)$ for a fixed $i < j$ then their
$\mm[i]$-composite is
$\zeta \cup \xi \colon w \mm[i] x \to z \mm[i] y$. The identity
$(j+1)$-cell on the $j$-cell $x$ is given by
$\emptyset \colon x \rightarrow x$. 
\end{Defn}


Note that this definition is identical to
Definition~\ref{def:orientals} except that $[n]_j$ is replaced by
$\dbr{n}_j$ and the meaning of $(\thg)^+$ and $(\thg)^-$ adapted
accordingly. This is because both are instances of the general
definition in~\cite{Street1991Parity} of the \emph{free
  $\omega$-category on a parity complex}; the basic data of a parity
complex are sets like $[n]_j$ or $\dbr{n}_j$ equipped with functions
$(\thg)^+$ and $(\thg)^-$ satisfying axioms. As before, it is quite
non-trivial that the cubes are well-defined $\omega$-categories, and
as before, we will be able to deduce this well-definedness from the
inductive argument we give.

As before, we refer to the conditions in (i) and (ii) above as
\emph{well-formedness} and \emph{movement}, and with the same
notational conventions. Exactly the same argument as in
Lemma~\ref{lem:1} now shows that:
\begin{Lemma}
  \label{lem:2}
  Any $1$-cell $\xi \colon a \rightarrow b$ in $\Q(n)$ is of the form
  $\xi = \{f_1, \dots, f_r\}$, where either $r = 0$ and $a = b$, or
  $r > 0$, $f_1^- = a$, $f_i^+ = f_{i+1}^-$ for all $1 < i < r$ and
  $f_r^+ = b$.
\end{Lemma}
With this in place, we are ready to give the proof of our second main
result, which follows a very similar pattern to the first.

\begin{Thm} \label{thm:cubes}
The $n$th cylinder $c^n(1)$ on the terminal $\omega$-category is isomorphic to the $n$th cube $\Q(n)$.
\end{Thm}
\begin{proof}
  First we introduce some notation. Given
 $a = a_1 \, \cdots \, a_n \in \dbr{n}$, we write $a\eta$ for
 $a_1 \, \cdots \, a_n\eta  \in \dbr{n+1}$ where $\eta \in \{\MM,\DD,\PP\}$, and given
 $\xi \subset \dbr{n}$, we write $\xi\eta$ for
 $\{a\eta : a \in \xi\}$. Note that for any $\xi \subset \dbr{n}_{j}$
 and $j > 0$, we have that:
\begin{equation}\label{eq:14}
(\xi\DD)^+ =
\begin{cases}
  (\xi^+)\DD \cup \xi\PP & \text{if $j$ even;}\\
  (\xi^+)\DD \cup \xi\MM & \text{if $j$ odd,}
\end{cases} \quad
(\xi\DD)^- =
\begin{cases}
  (\xi^-)\DD \cup \xi\MM & \text{if $j$ even;}\\
  (\xi^-)\DD \cup \xi\PP & \text{if $j$ odd,}
\end{cases}
\end{equation}
and that $({\xi \eta})^{\epsilon} = (\xi^{\epsilon}) \eta $ for any
$\eta \in \{\PP,\MM\}$ and $\epsilon \in \{+,-\}$. We now prove the
result by induction on $n$, simultaneously with the result that:
  \begin{equation}
    \label{eq:17}
    \text{if $\xi \colon x \rightarrow y$ and $\zeta \colon y \rightarrow z$ are $(j+1)$-cells in $\Q(n)$, then $\xi \cap \zeta = \emptyset$.}
  \end{equation}

  The case $n=0$ is clear. For the inductive step, we assume the
  result for $n$, and begin by showing $c(\Q(n)) \cong \Q(n+1)$.
  Recall that $c(\Q(n))$ is the collage of the bimodule
  $(\thg)/\Q(n)/(\thg)$ determined by bislice and thus contains two
  copies of $\Q(n)$ embedded on the left and right which we call
  $\Q(n)_l$ and $\Q(n)_r$. These can be mapped into $\Q(n+1)$ via
  $\omega$-functors $(\thg)\MM \colon \Q(n)_l \to \Q(n+1)$ and
  $(\thg)\PP \colon \Q(n)_r \to \Q(n+1)$ which are easily shown to be
  bijective on hom-$\omega$-categories and jointly bijective on
  $0$-cells. In this way, we determine all of the data for an
  $\omega$-isomorphism $\varphi \colon c(\Q(n)) \to \Q(n+1)$ except
  for the action on hom-$\omega$-categories
  $c(\Q(n))(a,b) = a / \Q(n) /b$ when $a \in \Q(n)_l$ and
  $b \in \Q(n)_r$. To give this action is equally to give
  $\omega$-isomorphisms
  \[\varphi_{a,b} \colon c(\Q(n))(a, b) = a / \Q(n) / b \to
    \Q(n+1)(a\MM, b\PP)\]
  which are compatible with composition, in the sense that we have
  $\varphi_{a,d}(k \Cdot \cell x \Cdot h) = \varphi_{c,d}(k) \mm[0]
  \varphi_{b,c}(\cell x) \mm[0] \varphi_{a,b}(h)$
  for each
  $(k, \cell x, h) \in \Q(n)(c, d) \times b / \O(n) / c \times
  \Q(n)(a, b) $.

  We will define $\varphi_{a,b}$ on cells of all dimension by:
\[
\phi_{a,b}(\cell x) =
\begin{cases}
{\bar x}\MM \cup {x}\DD \cup {{\hat x}}\PP \colon a\MM \to b\PP & \text{for $\cell x$ a $0$-cell;}\\
{{\bar x}}\MM \cup {x}\DD \cup {{\hat x}}\PP \colon
\phi_{a,b}({\cell m}) \to \phi_{a,b}({\cell p}) &
  \text{for $\cell x \colon \cell m \rightarrow \cell p$ a $(j+1)$-cell;}
\end{cases}
\]
for example, the action on $0$- and $1$-cells is as in the following diagram:
\begin{equation*}
\cd[@C+2em@R-0.5em@-.4em]{
 & 
m_0\MM \ar[dd]|{x\MM} \ar[r]^{m_0\DD} & 
m_0\PP \ar[dd]|{x\PP} \ar@/^.6em/[dr]^{{\hat m}_0\PP} & \\
a\MM \ar@/^.6em/[ur]^{{\bar m_0}\MM} \ar@/_.6em/[dr]_{{\bar p_0}\MM} 
\dtwocell[0.55]{r}{{\bar x}\MM} & 
\dtwocell[0.45]{r}{x\DD}  & 
\dtwocell[0.35]{r}{{\hat x}\PP} & 
b\PP \rlap{ .}\\
& 
p_0\MM \ar[r]_{p_0\DD} & 
p_0\PP \ar@/_.6em/[ur]_{{\hat p}_0\PP} & 
  }
\end{equation*}

We will show by induction on dimension that this assignation is
well-defined and bijective. For the base case, given a $0$-cell
$(x, \bar x, \hat x)$ of $a / \Q(n) / b$ we may write
$\bar x = \{f_1, \dots, f_r\} \colon a \rightarrow x$ and
$\hat x = \{g_1, \dots, g_s\} \colon x \rightarrow b$ with the
$f_i$'s and $g_k$'s satisfying the conditions of Lemma~\ref{lem:2};
since $({\xi \eta})^{\epsilon} = (\xi^{\epsilon}) \eta $ for any
$\eta \in \{\PP,\MM\}$ and $\epsilon \in \{+,-\}$, it follows that
\begin{equation*}
  \{f_1\MM, \dots, f_r\MM, x\DD, g_1\PP, \dots, g_s\PP\} \colon a \MM \rightarrow b \PP
\end{equation*}
is a well-defined $1$-cell of $\Q(n+1)$; in fact, it is easy to see
from Lemma~\ref{lem:2} that \emph{any}
$\xi \colon a\MM \rightarrow b\PP$ in $\Q(n+1)$ is of this form for a
unique $(x, \bar x, \hat x)$, and so $\varphi_{a,b}$ is not only
well-defined but also bijective on $0$-cells.

Suppose now that we have shown $\varphi_{a,b}$ is well-defined and
bijective on all cells up to dimension $j$; then for any parallel pair
of $j$-cells $(\cell m_j, \cell p_j)$ in $a / \O(n) / b$ with successive
boundaries $(\cell m_k, \cell p_k)$ for $k<j$, we will show that
$\phi_{a,b}$ gives a well-defined bijection between cells
$\cell x \colon \cell m_j \rightarrow \cell p_j$ and ones
$\xi \colon \phi_{a,b}(\cell m_j) \rightarrow \phi_{a,b}(\cell p_j)$. We
consider only the case where $j$ is odd; the even case is identical in
form, and so omitted. Observe first that the operation on subsets
\begin{equation}\label{eq:15}
  \begin{aligned}
    \P \dbr{n}_{j+1} \times \P \dbr{n}_{j+2} \times \P \dbr{n}_{j+2} & \rightarrow \P \dbr{n+1}_{j+2}  \\
    (x, \bar x, \hat x) & \mapsto {\bar x}\MM \cup x \DD \cup {\hat x}\PP
  \end{aligned}  
\end{equation}
underlying $\phi_{a,b}$'s action on $(j+1)$-cells is bijective. A cell
$\cell x \colon \cell m_j \rightarrow \cell p_j$ is an element
$(x, \bar x, \hat x)$ of the domain of~\eqref{eq:15} satisfying the four
conditions that:
\begin{align*}
  \text{(i)} & \text{  $x$, $\bar x$ and $\hat x$ are well-formed;} &
  \text{(iii)} & \text{  $\bar x \colon \bar m_j \rightarrowtriangle \bar p_j \cup x$;} \\ 
  \text{(ii)} & \text{  $x \colon m_j\rightarrowtriangle p_j$;} &
  \text{(iv)} & \text{ $\hat x \colon \hat m_j \cup x \rightarrowtriangle \hat p_j$,}
\end{align*}
while a cell
$\phi_{a,b}(\cell m_j) \rightarrow \phi_{a,b}(\cell p_j)$ is an element
${\bar x}\MM \cup x \DD \cup {\hat x}\PP$ of the codomain satisfying
the two conditions that:
\begin{align*}
\text{(v)} & \text{ ${\bar x}\MM \cup x \DD \cup {\hat x}\PP$ is well-formed;} \\ 
\text{(vi)} & \text{ ${\bar x}\MM \cup x \DD \cup {\hat x}\PP \colon 
{\bar m_j}\MM \cup m_j \DD \cup {\hat m_j}\PP \rightarrowtriangle 
{\bar p_j}\MM \cup p_j \DD \cup {\hat p_j}\PP$.}
\end{align*}
Thus to check well-definedness and bijectivity of $\varphi_{a,b}$ on
$(j+1)$-cells, it suffices to show that (i)--(iv) are equivalent to
(v) \& (vi). 

Now, if ${\bar x}\MM \cup x \DD \cup {\hat x}\PP$ is well-formed then
so are $\bar x$, $x$ and $\hat x$, since if $a \neq b \in x$ shared a
positive or negative face, then so would $a\DD \neq b\DD$ in
$x\DD\subset\xi$, and correspondingly for $\bar x$ and $\hat x$; so
(v) implies (i). Conversely, if $x$, $\bar x$ and $\hat x$ are
well-formed, then each component of
${\bar x}\MM \cup x \DD \cup {\hat x}\PP$ is individually well-formed,
and so it remains to check the cross-terms. First, $a\in{{\bar x}}\MM$
and $b\in{{\hat x}}\PP$ cannot share \emph{any} face, since its final
symbol would be $\MM$ and $\PP$ simultaneously. Next, if
$a\in{{\bar x}}\MM$ and $b\in{x}\DD$ then
$a^- \subset {{\bar x}^-}\MM$ and $b^- \subset {x}\MM \cup {x^-}\DD$;
but (iii) ensures that ${\bar x}^-$ and $x$ are disjoint so
$a^- \cap b^- =\emptyset$. Likewise $a^+ \subset {{\bar x}^+}\MM$ and
$b^+ \subset {x}\PP \cup {x^+}\DD$ and so $a^+ \cap b^+ = \emptyset$.
A similar argument shows that $a\in{{\hat x}}\PP$ and $b\in{x}\DD$
cannot share an odd face or an even face, and so (i) and (iii) imply (v).

Turning now to the movement conditions, we have 
$(x\DD)^+ = x^+\DD \cup x\PP$ and
$(x\DD)^- = x^-\DD \cup x\MM$ by~\eqref{eq:14}; so (vi)
is equivalent to:
 \begin{align*}
{\bar p}_j\MM \cup p_j\DD \cup {\hat p}_j\PP & = 
({\bar m}_{j}\MM \cup m_j\DD \cup {\hat m}_j\PP) \cup
({\bar x}^+\MM \cup x^+\DD \cup x\PP \cup {\hat x}^+\PP) \\
& \hspace{13em} \setminus ({\bar x}^-\MM \cup x^-\DD \cup x\MM \cup {\hat x}^-\PP) \\
{\bar m}_j\MM \cup m_j\DD \cup {\hat m}_j\PP & = 
({\bar p}_{j}\MM \cup p_j\DD \cup {\hat p}_j\PP) \cup
({\bar x}^-\MM \cup x^-\DD \cup x\MM \cup {\hat x}^-\PP) \\
& \hspace{13em} \setminus ({\bar x}^+\MM \cup x^+\DD \cup x\PP \cup {\hat x}^+\PP) 
\rlap{ ;}
 \end{align*}
 now as terms ending with the three possible symbols are disjoint, and
 each $(\thg)\eta$ for $\eta\in\{\DD,\MM,\PP\}$ is a bijection, this
 is equivalent to the six conditions:
 \begin{equation}\label{eq:16}
 \begin{aligned}
   p_j & = m_j \cup x^+ \setminus x^- & \qquad 
   m_j & = p_j \cup x^- \setminus x^+ \\
   {\bar p_j} & = {\bar m_j} \cup {\bar x}^+ \setminus ({\bar x}^- \cup x) & \qquad 
   {\bar m_j} & = \bar p_j \cup {\bar x}^- \cup x \setminus {\bar x}^+ \\
   {\hat p_j} & = {\hat m_j} \cup {\hat x}^+  \cup x\setminus {\hat x}^- & \qquad 
   {\hat m_j} & = \hat p_j \cup {\hat x}^- \setminus ({\hat x}^+ \cup x) 
   \rlap{ .}
\end{aligned}
\end{equation}
The top row is precisely (ii), the middle right is the second movement
condition for (iii), and the bottom left is the first movement
condition for (iv). The middle left will imply the first movement
condition
$\bar p_j \cup x = \bar m_j \cup {\bar x}^+ \setminus \bar x^-$ for
(iii) so long as $\bar x^- \cap x = \emptyset$; but this is certainly
the case if ${\bar x \MM} \cup {x\DD} \cup {\hat x \PP}$ is
well-formed, as
$(\bar x^- \cap x)\MM = \bar x^-\MM \cap x\MM \subset (\bar x\MM)^-
\cap (x\DD)^- = \emptyset$.
The bottom right will imply the first second condition
$\hat m_j \cup x = \hat p_j \cup {\hat x}^- \setminus \hat x^+$ for
(iv) so long as $\hat x^+ \cap x = \emptyset$; again, this is the case
if ${\bar x \MM} \cup {x\DD} \cup {\hat x \PP}$ is well-formed, as
$(\hat x^+ \cap x)\PP = \bar x^+\PP \cap x\PP \subset (\hat x\PP)^+
\cap (x\DD)^+ = \emptyset$. So (v) and (vi) imply (ii)--(iv).

Finally, since (vi) is equivalent to the conditions in~\eqref{eq:16},
it will follow from (ii)--(iv) so long as we know that $x \cap
\bar p_j = \emptyset$ and $\hat m_j \cap x = \emptyset$. 
Now, observe that $\bar x$ and $\hat x$ are cells
\begin{gather*}
  \qquad \ \bar m_j \rightarrow \bar p_j \mm[j] \bar p_{j-2} \mm[j-2] \cdots \bar p_1 \mm[1] x \mm[0] \bar m_0  \cdots \mm[j-1] \bar m_{j-1}\\
\text{and} \quad   \hat p_{j-1} \mm[j-1] \cdots \hat p_0 \mm[0] x \mm[1] \hat m_1  \cdots \mm[j-2] \hat m_{j-2} \mm[j] \hat m_{j}  \rightarrow \hat p_j 
\end{gather*}
so that in particular, the cells $\bar p_j$ and
$\bar p_{j-2} \mm[j-2] \cdots \bar p_1 \mm[1] x \mm[0] \bar m_0 \cdots
\mm[j-1] \bar m_{j-1}$
of $\Q(n)$ are $j$-composable; applying the inductive instance
of~\eqref{eq:17} for $\Q(n)$ we conclude that $x \cap \bar p_j = \emptyset$.
The same argument, applied to the domain of $\hat x$, shows that also $\hat m_j \cap x = \emptyset$ as required.

This shows that each $\varphi_{a,b}$ is a bijective map on cells of
all dimension; we next show $\omega$-functoriality and compatibility
with composition. If $\cell x$ and $\cell y$ are $(j+1)$-cells of
$a /\O(n) / b$ with $t_\ell(\cell x) = s_\ell(\cell y)$ and with
common boundary $(\cell m_k, \cell p_k)$ for each $k < \ell$, then:
  \begin{align*}
  \phi_{a,b}(\cell y \mm[\ell] \cell x)
  & = \phi_i(y \mm[\ell] x,\, \bar y \mm[\ell+1] \bar p_{\ell-1} \mm[\ell-1] \cdots \bar p_1 \mm[1] s_{\ell+1}y \mm[0] \bar m_0 \cdots \mm[\ell-2] \bar m_{\ell-2} \mm[\ell] {\bar x},\\
  & \hspace{6em} \hat y \mm[\ell] \hat p_{\ell-2} \mm[\ell-2] \cdots \hat p_0 \mm[0] t_{\ell+1}x \mm[1] \hat m_1 \cdots \mm[\ell-1] \hat m_{\ell-1} \mm[\ell+1] {\hat x})
   \\
  & = \varphi_{a,b}(y \cup x, {\bar y} \cup {\bar x}, {\hat y} \cup {\hat x} ) 
   = (y \cup x)\DD \cup (\bar y \cup {\bar x})\MM \cup (\hat y \cup \hat x)\PP \\ 
  & = ({y\DD} \cup {{\bar y}\MM} \cup {{\hat y}\PP}) \cup ({x\DD} \cup {{\bar x}\MM} \cup {{\hat x}\PP}) 
   = \phi_{a,b}(\cell y) \mm[\ell] \phi_{a,b}(\cell x)
  \end{align*}
 when $\ell$ is even, and correspondingly when $\ell$ is odd. 
 As for identity morphisms, we have
 $\phi_{a,b}(i\cell x) = \phi_{a,b}(ix, i{\bar x}, i{\hat x}) = \emptyset\MM \cup \emptyset\DD \cup\emptyset\PP = \emptyset = i\phi_{a,b}(\cell x)$;
 so $\phi_{a,b}$ is $\omega$-functorial as required. Compatibility
 of the $\varphi_{a,b}$'s with composition is similar: we have that
 $ \phi_{a,d}(k \Cdot \cell x \Cdot h) = \phi_{a,d}(x, {\bar x} \mm[0] h, k \mm[0] {\hat x}) = ({\bar x} \cup h)\MM \cup x\DD \cup (k \cup {\hat x})\PP = k \PP \cup ({\bar x}\MM \cup x\DD \cup {\hat x}\PP) \cup h\MM = \phi_{c,d}(k)\mm[0] \phi_{b,c}(\cell x) \mm[0] \phi_{a,b}(h)$.

 This proves that $c(\Q(n)) \cong \Q(n+1)$, and it remains only to
 derive~\eqref{eq:17} for $\Q(n+1)$. In the case of $0$-composable
 $1$-cells, we see from Lemma~\ref{lem:2} that if
 $\{f_1, \dots, f_r\} \colon a \rightarrow b$ is a $1$-cell of
 $\Q(n+1)$, then each $f_i$ will contain at least as many $\PP$'s as
 $a$ and strictly fewer $\PP$'s than $b$; so if
 $\{g_1, \dots, g_s\} \colon b \rightarrow c$ is another $1$-cell,
 then each $g_k$ must contain strictly more $\PP$'s than each $f_i$,
 thus proving disjointness. As for $(j+1)$-composable $(j+2)$-cells in
 $\Q(n+1)$, any pair of such must live in some hom-$\omega$-category;
 the only new case to consider is that of
 $\Q(n+1)(a\MM, b\PP) \cong a / \Q(n) / b$. So let
 $\cell x \colon \cell a \rightarrow \cell b$ and
 $\cell y \colon \cell b \rightarrow \cell c$ be $(j+1)$-cells in
 $a / \Q(n) / b$ with common boundary $(\cell m_k, \cell p_k)$ for all
 $k < j$. We have $x \colon a \rightarrow b$ and
 $y \colon b \rightarrow c$ in $\Q(n)$, whence $x \cap y = \emptyset$
 by~\eqref{eq:17} for $\Q(n)$; moreover, assuming $j$ is even, we have
 that $\bar y$ and
 $\bar p_{j-1} \mm[j-1] \cdots \bar p_1 \mm[1] y \mm[0] \bar m_0
 \cdots \mm[j-2] \bar m_{j-2} \mm[j] {\bar x}$
 are $(j+1)$-composable $(j+2)$-cells, whence
 $\bar x \cap \bar y = \emptyset$ again by~\eqref{eq:17} for $\Q(n)$;
 similarly, $\hat x$ and
 $\hat y \mm \hat p_{j-2} \mm[j-2] \cdots \hat p_0 \mm[0] x \mm[1]
 \hat m_1 \cdots \mm[j-1] \hat m_{j-1}$
 are $(j+1)$-composable $(j+2)$-cells, whence
 $\hat x \cap \hat y = \emptyset$. A dual argument applies when $j$ is
 odd, and in both cases we conclude that the composable pair
 $\phi_i(\cell x) = \bar x\MM \cup x\DD \cup \hat x\PP$ and
 $\phi_i(\cell y) = \bar y\MM \cup y\DD \cup \hat y\PP$ satisfy
 $(\bar x\MM \cup x\DD \cup \hat x\PP) \cap (\bar y\MM \cup y\DD \cup
 \hat y\PP) = \emptyset$, as required.
\end{proof}
As before, we have avoided using any aspect of the theory of parity
complexes beyond the basic definitions, and so in an identical manner
to Remark~\ref{rk:1} we may exploit the preceding theorem
to give a simpler proof that the cubes are indeed well-defined
$\omega$-categories.

\appendix
\section{Proofs of well-definedness}\label{sec:proofs-well-defin}

\begin{Prop}
  \label{prop:1}
  For any $\omega$-category $C$ and $a \in C_0$, the lax coslice $a/C$
  is a well-defined $\omega$-category.
\end{Prop}
\begin{proof}
  We first show by induction on $n$ that (a) the cells of
  $a/C$ of dimension $\leqslant n$ are well-defined; and (b) for any
  parallel pair of $n$-cells $(\cell m_n, \cell p_n)$ with $i$-boundary
  $(\cell m_i, \cell p_i)$ for all $i < n$, there is, for $n$ even,
  a well-defined $\omega$-functor
  \begin{equation}\label{eq:1}
    M_n \colon C(m_0, p_0) \cdots (m_n, p_n) \rightarrow C(s\bar p_0, t\bar p_0)(s\bar m_1, t\bar m_1) \cdots (s\bar p_n, t\bar p_n)
  \end{equation}
  sending $x$ to
  $\bar p_{n-1} \mm[n-1] \cdots \bar p_3 \mm[3] \bar p_1 \mm[1] x
  \mm[0] \bar m_0 \mm[2] \bar m_2 \cdots \mm \bar m_n$,
  and, for $n$ odd, a well-defined $\omega$-functor
  \begin{equation}\label{eq:2}
    P_n \colon C(m_0, p_0) \cdots (m_n, p_n) \rightarrow C(s\bar p_0, t\bar p_0)(s\bar m_1, t\bar m_1) \cdots (s\bar m_n, t\bar m_n)
  \end{equation}
  sending $x$ to
  $\bar p_n \mm \cdots \bar p_3 \mm[3] \bar p_1 \mm[1] x \mm[0]
  \bar m_0 \mm[2] \bar m_2 \cdots \mm[n-1] \bar m_{n-1}$.

  For the base case $n = 0$, it is clear for (a) that the notion of
  $0$-cell is well-defined. As for (b), if $(\cell m_0, \cell p_0)$ are
  a (necessarily parallel) pair of $0$-cells, then, since
  $\bar m_0 \colon a \rightarrow m_0$, the assignation
  $x \mapsto x \mm[0] \bar m_0$ defines an $\omega$-functor
  $M_0 \colon C(m_0, p_0) \rightarrow C(s\bar p_0, t \bar p_0) = C(a,
  p_0)$ as required for~\eqref{eq:1}.

  We now assume the result for $n$, and verify it for $(n+1)$. First
  let $n$ be even. For (a), let $(\cell m_n, \cell p_n)$ be a parallel
  pair of $n$-cells, and $M_n$ the associated
  $\omega$-functor~\eqref{eq:1}; then an $(n+1)$-cell
  $\cell x \colon \cell m_n \rightarrow \cell p_n$ of $a/C$ is a pair
  \begin{equation}\label{eq:3}
    (x \in C(m_0, p_0) \cdots (m_n, p_n),\, \bar x \colon M_nx \rightarrow \bar p_n)\rlap{ ,}
  \end{equation}
  and so well-defined. For (b), if
  $\cell m_{n+1}, \cell p_{n+1}$ are both $(n+1)$-cells $\cell m_n \rightarrow \cell
  p_n$,
  then $\bar m_{n+1} \colon M_nm_{n+1} \rightarrow \bar p_n$ and
  $\bar p_{n+1} \colon M_np_{n+1} \rightarrow \bar p_n$; whence the
  assignation $x \mapsto \bar p_{n+1} \mm[n+1] M_nx$ yields an
  $\omega$-functor
  \begin{equation*}
    C(m_0, p_0) \cdots (m_{n+1}, p_{n+1}) \rightarrow C(s\bar p_0, t\bar p_0)\cdots (s\bar p_{n}, t\bar p_{n+1})(M_nm_{n+1}, \bar p_n) \rlap{ ,}
  \end{equation*}
  which 
  is of the correct form to be the $P_{n+1}$ of~\eqref{eq:2}. Suppose
  now that $n$ is odd. For (a), if $(\cell m_n, \cell p_n)$ are
  parallel $n$-cells, and now $P_n$ is the associated $\omega$-functor
  of~\eqref{eq:2}, then an $(n+1)$-cell
  $\cell x \colon \cell m_n \rightarrow \cell p_n$ of $a/C$ is a pair
  \begin{equation}\label{eq:4}
    (x \in C(m_0, p_0) \cdots (m_n, p_n),\, \bar x \colon \bar m_n \rightarrow P_nx)\rlap{ ,}
  \end{equation}
  and so, again, well-defined. For (b), if
  $\cell m_{n+1}, \cell p_{n+1} \colon \cell m_n \rightarrow \cell
  p_n$,
  then the operation $x \mapsto P_nx \mm[n+1] \bar m_{n+1}$ defines an
  $\omega$-functor of the right form to be the $M_{n+1}$
  of~\eqref{eq:1}. This completes the inductive step.

  So $a/C$ is well-defined as a globular set; given
  $\cell x \in (a/C)_k$ and $n < k$, we will denote the
  $\omega$-functor~\eqref{eq:1} or~\eqref{eq:2} associated to the
  $n$-boundary $(\cell m_n, \cell p_n)$ of $\cell x$ as
  $M_{n}^\cell x$ (for $n$ even) or $P_n^\cell x$ (for $n$ odd).
  Note that, for each $n < k$, we have by~\eqref{eq:3}, \eqref{eq:4}
  and induction that:
  \begin{equation}\label{eq:5}
    s_{n}(\bar x) =
    \begin{cases}
      \bar m_{n-1} & \text{$n$ even;}\\
      M_{n-1}^\cell{x}(m_{n}) & \text{$n$ odd,}
    \end{cases} \ \  \text{and} \ \ \ 
    t_{n}(\bar x) =
    \begin{cases}
      P_{n-1}^\cell{x}(p_{n}) & \text{$n$ even;}\\
      \bar p_{n-1} & \text{$n$ odd.}
    \end{cases}
  \end{equation}

  We now show that $a/C$ is a well-defined
  $\omega$-category. The identity operations are clearly well-defined;
  for composition, let
  $\cell x \colon \cell a \rightsquigarrow \cell b$ and
  $\cell y \colon \cell b \rightsquigarrow \cell c$ be $n$-composable
  $k$-cells whose common $i$-boundary for $i < n$ is
  $(\cell m_i, \cell p_i)$. First let $n$ be odd. Writing
  $M = M_{n-1}^{\cell x} = M_{n-1}^{\cell y}$, the composite cell
  in $a/C$ is the pair
  \begin{equation*}
    \cell y \mm \cell x \defeq (y \mm x,\, \bar y \mm M(t_{n+1}x) \mm[n+1] \bar x)\rlap{ .}
  \end{equation*}
  The first component is clearly well-defined; writing
  $\bar y \ast \bar x$ for the second, note that the $\omega$-functors
  $P^{\cell x}_n$ and $P^{\cell y}_n$ satisfy
  $P^{\cell x}_n(u) = \bar b \mm M(u)$ and
  $P^{\cell y}_n(u) = \bar c \mm M(u)$; from this and~\eqref{eq:5} we
  conclude that
  \begin{equation*}
    \bar x \colon \bar a \rightsquigarrow \bar b \mm M(t_{n+1}x) 
     \qquad \text{and} \qquad 
    \bar y \colon \bar b \rightsquigarrow \bar c \mm M(t_{n+1}y)\rlap{ ,} 
  \end{equation*}
  so that $\bar y \ast \bar x$ is indeed a
  well-defined cell of $C$. We next check it has the correct source
  and target. If $k = n+1$, then we should have
  $\cell y \mm \cell x \colon \cell a \rightarrow \cell c$; so
  by~\eqref{eq:5}, $\bar y \ast \bar x$ should be
  a map $\bar a \rightarrow \bar c \mm M(y \mm[n] x)$. But this is so
  since it is the composite
  \begin{equation*}
    \bar a \xrightarrow{\bar x} \bar b \mm Mx \xrightarrow{\bar y \mm Mx} \bar c \mm My \mm Mx =
\bar c \mm M(y \mm x)\rlap{ .}
  \end{equation*}
  Now let $k>n+1$; if $\cell x \colon \cell u \rightarrow \cell v$ and
  $\cell y \colon \cell w \rightarrow \cell z$, then we should have
  $\cell y \mm \cell x \colon \cell w \mm \cell u \rightarrow \cell z
  \mm \cell v$.
  We show by induction on $k$ that (a) $\cell y \mm \cell x$ is a
  $k$-cell of this form; and (b) for all cells
  $f \colon u \rightsquigarrow v$ and
  $g \colon w \rightsquigarrow z$ we have:
  \begin{equation}\label{eq:6}
    \begin{aligned}
      P^{\cell y \mm \cell x}_{k-1}(g \mm f) &= P^{\cell y}_{k-1}(g) \mm M(t_{n+1}x) \mm[n+1] P^{\cell x}_{k-1}(f) &&\text{ if $k$ even;}\\
      M^{\cell y \mm \cell x}_{k-1}(g \mm f) &= M^{\cell y}_{k-1}(g) \mm M(t_{n+1}x) \mm[n+1] M^{\cell x}_{k-1}(f) &&\text{ if $k$ odd.}
    \end{aligned}
  \end{equation}

  Assuming (a) and (b) for all $j < k$, we prove it for $k$. If $k$ is
  odd, then by~\eqref{eq:5} we have
  $\bar x \colon M^{\cell x}_{k-1}(x) \rightarrow \bar v$ and
  $\Bar y \colon M^{\cell y}_{k-1}(y) \rightarrow \bar z$ and require
  for (a) that
  $\bar y \ast \bar x \colon M'(y \mm x) \rightarrow \bar z \ast \bar
  v$,
  where $M'$ is the $\omega$-functor associated to the parallel pair
  $(\cell w \mm \cell u, \cell z \mm \cell v)$. Note first that we
  have
  $t(\bar y \ast \bar x) = t(\bar y \mm M(t_{n+1}x) \mm[n+1] \bar x) =
  t\bar y \mm Mv \mm[n+1] t\bar x = \bar z \mm M(t_{n+1}v) \mm[n+1]
  \bar v = \bar z \ast \bar v$
  as required; on the other hand, we have
  $s(\bar y \ast \bar x) = s\bar y \mm M(t_{n+1}x) \mm[n+1] s\bar x =
  M^{\cell y}_{k-1}(y) \mm M(t_{n+1}x) \mm[n+1] M^{\cell x}_{n+1}(x)$
  so that for $\cell y \mm \cell x$ to be a cell of the form required
  for (a), it will suffice to prove
  \begin{equation}
    \label{eq:7}
  M'(g \mm f) = M^{\cell y}_{k-1}(g) \mm M(t_{n+1}x) \mm[n+1] M^{\cell
    x}_{k-1}(f)
  \end{equation}
  for all cells $f \colon u \rightsquigarrow v$ and
  $g \colon w \rightsquigarrow z$. Once we know $\cell y \mm \cell x$
  is a cell, we will have $M' = M^{\cell y \mm \cell x}_{k-1}$, so
  that~\eqref{eq:7} gives (b) as required. We verify~\eqref{eq:7}
  first when $k = n+2$; here, \eqref{eq:1}, functoriality of
  $M$ and interchange gives
  \begin{align*}
    M'(g \mm f) &= \bar c \mm M(g \mm f) \mm[n+1] (\bar w \ast \bar u) = 
    \big(\bar c \mm M(g \mm f)\big) \mm[n+1] (\bar w \mm Mu) \mm[n+1] \bar u\\
    &= \big(\bar c \mm M(g \mm v)\big) \mm[n+1] \big(\bar c \mm M(w \mm f)\big) \mm[n+1] (\bar w \mm Mu) \mm[n+1] \bar u\\
    &= \big(\bar c \mm M(g \mm v)\big) \mm[n+1] \big(\bar w \mm M(w \mm f)\big) \mm[n+1] \bar u\\
    &= \big(\bar c \mm M(g \mm v)\big) \mm[n+1] (\bar w \mm Mv) \mm[n+1] (\bar b \mm Mf) \mm[n+1] \bar u\\
    &= (\bar c \mm Mg \mm[n+1] \bar w) \mm Mv \mm[n+1] (\bar b \mm Mf \mm[n+1] \bar u)\\
    &= M^{\cell y}_{k-1}(g) \mm M(t_{n+1}x) \mm[n+1] M^{\cell x}_{k-1}(f)
  \end{align*}
  as required. In the case $k > n+2$, we have that
  \begin{align*}
    M'(g \mm f) &= P^{\cell w \mm \cell u}_{k-2}(g \mm f) \mm[k-1] (\bar w \ast \bar u) \\
    &= \big(P^{\cell w}_{k-2}(g) \mm M(t_{n+1}x) \mm[n+1] P^{\cell u}_{k-2}(f)\big)\mm[k-1] (\bar w \mm Mu \mm[n+1] \bar u)\\
    &= (P^{\cell w}_{k-2}(g) \mm[k-1] \bar w) \mm M(t_{n+1}x) \mm[n+1] (P^{\cell u}_{k-2}(f) \mm[k-1] \bar u)\\
    &= M^{\cell y}_{k-1}(g) \mm M(t_{n+1}x) \mm[n+1] M^{\cell x}_{k-1}(f)
  \end{align*}
  by~\eqref{eq:2}, the case $(k-1)$ of~\eqref{eq:6}, and 
  interchange. This completes the inductive step for odd $k$; we omit
  the analogous argument for $k$ even.

  We have thus proved for odd $n$ that composition $\mm$ in $a/C$ is
  well-defined and satisfies the source--target axioms; the case where
  $n$ is even is analogous, and so omitted. The identity axioms for
  $a/C$ are easy; next, for associativity, we must show that
  $\cell x \mm (\cell y \mm \cell z) = (\cell x \mm \cell y) \mm \cell
  z$
  in $a/C$. Suppose that $n$ is odd, and let
  $M = M_{n-1}^{\cell x} = M_{n-1}^{\cell y} = M_{n-1}^{\cell z}$.
  Then the two iterated composites are
  \begin{align*}
    \big((x \mm y) \mm z,\,&\  (\bar x \mm M(t_{n+1}y) \mm[n+1] \bar y) \mm M(t_{n+1}z) \mm[n+1] \bar z\big)\\
    \text{and} \ 
    \big(x \mm (y \mm z),\,&\  \bar x \mm M(t_{n+1}(y \mm z)) \mm[n+1] (\bar y \mm M(t_{n+1}z) \mm[n+1] \bar z)\big)
\end{align*}
which are easily equal by functoriality of $M$ and 
interchange. The case of $n$ even is dual, and so omitted; and it
remains only to verify the interchange axiom
$(\cell z \mm[k] \cell w) \mm (\cell y \mm[k] \cell x) = (\cell z \mm
\cell y) \mm[k] (\cell w \mm \cell x)$
for all suitable cells $\cell x, \cell y, \cell w, \cell z$ in $a/C$.
Of course, the equality is clear on first components; on second
components, there are four cases to consider depending on the parities
of the dimensions $n < k$; we give only the case where both $n$ and
$k$ are odd, as the others are similar. So let
$M = M^{\cell a}_{n-1} = M^{\cell b}_{n-1} = M^{\cell c}_{n-1}$, let
$M' = M^{\cell w}_{k-1} = M^{\cell z}_{k-1}$ and let
$M'' = M^{\cell x}_{k-1} = M^{\cell y}_{k-1}$. The second component of
$(\cell z \mm[k] \cell w) \mm (\cell y \mm[k] \cell x)$ is:
\begin{align*}
 & (\bar z \mm[k] M't_{k+1}w \mm[k+1] \bar w) \mm Mt_{n+1}(y \mm[k] x) \mm[n+1] (\bar y \mm[k] M''t_{k+1} x \mm[k+1] \bar x)\\
 =\ & \big(\,[(\bar z \mm[k] M't_{k+1}w) \mm Mt_{n+1}x] \mm[k+1] [\bar w \mm Mt_{n+1}x]\,\big) \mm[n+1] \big(\,[\bar y \mm[k] M''t_{k+1} x] \mm[k+1] \bar x\,\big)\\
 =\ & \big(\,[(\bar z \mm[k] M't_{k+1}w) \mm Mt_{n+1}x] \mm[n+1] [\bar y \mm[k] M''t_{k+1} x]\,\big) \mm[k+1] \big(\,[\bar w \mm Mt_{n+1}x] \mm[n+1] \bar x\,\big)
\end{align*}
using interchange. The left-hand bracketed term is in turn equal to
\begin{align*}
   & ([\bar z \mm Mt_{n+1}x] \mm[k] [M't_{k+1}w \mm Mt_{n+1}x]) \mm[n+1] (\bar y \mm[k] M''t_{k+1} x)\\
 =\ & ([\bar z \mm Mt_{n+1}x] \mm[n+1] \bar y) \mm[k] ([M't_{k+1}w \mm Mt_{n+1}x] \mm[n+1] M''t_{k+1} x)\\ =\ & ([\bar z \mm Mt_{n+1}x] \mm[n+1] \bar y) \mm[k] M^{\cell w \mm \cell x}_{k-1}(t_{k+1}(w \mm x))
\end{align*}
using interchange and~\eqref{eq:6}, which on recomposing with the
right-hand bracketed term above yields the second component of $(\cell
z \mm \cell y) \mm[k] (\cell w \mm \cell x)$, as required.
\end{proof}

\begin{Prop}
  \label{prop:5}
  For any $\omega$-category $C$ and objects $a,b \in C_0$, the
  $\omega$-functor $\Cdot \colon a/C \times C(b,a) \rightarrow b/C$ is well-defined.
\end{Prop}
\begin{proof}
  Recall that $\Cdot$ is defined on $0$-cells by
  $\cell x \Cdot h = (x, {\bar x} \mm[0] h)$ and on $(n+1)$-cells by
  $\cell x \Cdot h = (x, {\bar x} \mm[0] h) \colon s\cell x \Cdot sh
  \rightarrow t\cell x \Cdot th$.
  Well-definedness is clear on $0$-cells. At higher dimensions, we
  show by induction on $n$ that for each pair $(\cell x, h)$ of
  dimension $(n+1)$, the cell $\cell x \Cdot h$ is well-defined and
  satisfies
  \begin{equation}\label{eq:8}
    M^{\cell x \!\Cdot\! h}_{n}(\thg) = M^{\cell x}_{n}(\thg) \mm[0] sh \qquad \text{or} \qquad 
    P^{\cell x \!\Cdot\! h}_{n}(\thg) = P^{\cell x}_{n}(\thg) \mm[0] th
  \end{equation}
  according as $n$ is even or odd, where, as before, $M^{\cell x}_{n}$
  and $P^{\cell x}_{n}$ denote the auxiliary functors~\eqref{eq:1}
  and~\eqref{eq:2} associated to the $n$-boundary
  $(\cell m_n, \cell p_n)$ of $\cell x$.

  So let $\cell x \colon \cell m_n \rightarrow \cell p_n$ and
  $h \colon u \rightarrow v$ be $(n+1)$-cells of $a/C$ and $C(b,a)$;
  by induction $\cell m_n \Cdot u$ and $\cell p_n \Cdot v$ are
  well-defined, and we must show that
  $\cell x \Cdot h \colon \cell m_n \Cdot u \rightarrow \cell p_n
  \Cdot v$
  is too. Even without knowing this, we may still verify~\eqref{eq:8}
  since $M^{\cell x \!\Cdot\! h}_n$ or $P^{\cell x \!\Cdot\! h}_n$ (as
  the case may be) depend only on the well-defined boundary pair
  $(\cell m_n \Cdot u, \cell p_n \Cdot v)$. But when $n$ is even we
  have
  \begin{align*}
    M^{\cell x}_n(\thg) \mm[0] u & = (P^{\cell m_n}_{n-1}(\thg) \mm[n] \bar m_n) \mm[0] u 
                              = (P^{\cell m_n}_{n-1}(\thg) \mm[n] \bar m_n) \mm[0] (tu \mm[n] u) \\
                              & = (P^{\cell m_n}_{n-1}(\thg) \mm[0] tu) \mm[n] (\bar m_n \mm[0] u) 
                              = P^{\cell m_n \!\Cdot\! u}_{n-1}(\thg) \mm[n] (\bar m_n \mm[0] u)\\
                              &= M^{\cell x \!\Cdot h}_n(\thg)
  \end{align*}
  as required, and correspondingly for $n$ odd. We now use this to
  show that $\cell x \Cdot h = (x, \bar x \mm[0] h)$ is a well-defined
  cell $\cell m_n \Cdot u \rightarrow \cell p_n \Cdot v$. Clearly the
  first component is a map $x \colon m_n \rightarrow p_n$ as required.
  For the second component, suppose first that $n$ is even; then
  by~\eqref{eq:3}, $\bar x$ is a cell
  $M^{\cell x}_n(x) \rightarrow \bar p_n$, whose $0$-source is
  by~\eqref{eq:5} equal to $a$. Thus $\bar x \mm[0] h$ is a
  well-defined cell
  $M^{\cell x}_n(x) \mm[0] u \rightarrow \bar p_n \mm[0] v$ and by the
  above calculation $M^{\cell x}_n(x) \mm[0] u = M^{\cell x \!\Cdot
    h}_n(x)$ as required. The case where $n$ is odd is similar.
\end{proof}

\bibliographystyle{acm}
\bibliography{bibdata}

\end{document}